\documentclass[11pt,3p]{elsarticle}

\usepackage{setspace}
\usepackage[hidelinks]{hyperref}


\usepackage{graphicx}
\usepackage{amsmath,amssymb,amsfonts}
\usepackage[title]{appendix}
\usepackage{algorithm}
\usepackage{algpseudocode}

\newdefinition{definition}{Definition}
\newtheorem{theorem}{Theorem}
\newtheorem{proposition}[theorem]{Proposition}
\newtheorem{lemma}{Lemma}
\newtheorem{corollary}{Corollary}
\newproof{proof}{Proof}

\newcommand{\up}[1]{^{\mathrm{#1}}}
\usepackage{subcaption}
\usepackage{pgfplots}
\pgfplotsset{compat=1.9}
\usepgfplotslibrary{fillbetween}
\usepackage{tikz}
\usepackage{tikzscale}
\usetikzlibrary{patterns}
\usetikzlibrary{datavisualization.formats.functions}
\usetikzlibrary{decorations.pathreplacing,calligraphy}

\usepackage{makecell}
\usepackage{array}
\newcolumntype{C}[1]{>{\centering\arraybackslash}p{#1}}
\usepackage{subcaption}
\usepackage[export]{adjustbox}
\newcolumntype{M}[1]{>{\centering\arraybackslash}m{#1}}

\begin{document}

\title{How long is long enough? Finite-horizon approximation of energy storage scheduling problems} 

\author[a]{Eléa Prat\corref{cor1}}\ead{elea.prat@gmail.com}
\author[a]{Richard M. Lusby}\ead{rmlu@dtu.dk}
\author[b]{Juan Miguel Morales}\ead{juan.morales@uma.es}
\author[c]{Salvador Pineda}\ead{spineda@uma.es}
\author[d,a,e]{Pierre Pinson}\ead{p.pinson@imperial.ac.uk}

\cortext[cor1]{Corresponding author}

\affiliation[a]{organization={Department of Technology, Management and Economics, Technical University of Denmark}, city={Lyngby}, country={Denmark}}
\affiliation[b]{organization={Department of Mathematical Analysis, Statistics and Operations Research \& Applied Mathematics, Universidad de Málaga}, city={Málaga}, country={Spain}}
\affiliation[c]{organization={Department of Electrical Engineering, Universidad de Málaga, Universidad de Málaga}, city={Málaga}, country={Spain}}
\affiliation[d]{organization={Dyson School of Design Engineering, Imperial College London}, city={London}, country={UK}}
\affiliation[e]{organization={CoRE, Aarhus University}, city={Aarhus}, country={Denmark}}

\begin{abstract}
Energy storage scheduling problems, where a storage system is operated to maximize its profit in response to a price signal, are naturally formulated as infinite-horizon optimization problems, since storage systems operate continuously, without a foreseen end to their operation. Such problems can be solved to optimality with a rolling-horizon approach, provided that the planning horizon over which the problem is solved is long enough. Such a horizon is termed a \emph{forecast horizon}. Despite its importance, the planning horizon is usually chosen arbitrarily for such applications. We introduce an easy-to-check condition that confirms whether a planning horizon is a forecast horizon, and which can be used to derive a bound on suboptimality when it is not the case. 
In practice, this condition enables practitioners to evaluate whether the planning horizon chosen is sufficient, thereby providing, for the first time, a practical means of evaluating and selecting planning horizons for energy storage scheduling problems.
We also derive a lower bound on the minimum forecast horizon.
Building on the theoretical results, we develop an algorithm to determine the minimum forecast horizon. It enables the identification, a posteriori, of the shortest planning horizon that guarantees optimal rolling-horizon decisions while avoiding unnecessary forecasting effort and computational cost.
Numerical experiments illustrate the practical use of the proposed framework and investigate how storage system characteristics and electricity price patterns influence the minimum forecast horizon.
\end{abstract}


\maketitle

\newpage

\section{Introduction}\label{sec:intro}

Energy storage systems are key enablers in the transition to decarbonized energy systems. They guarantee the reliable operation of networks in which a large share of the production, coming from renewable sources, is variable and uncertain. For these reasons, their presence in power systems is expected to increase significantly in the coming years. As an example, worldwide investment in battery storage has doubled year-on-year since 2020 \citep{IEA2024Invest}. Storage systems with different durations of charge and discharge -- which corresponds to the time they need to fully charge or discharge -- are required in order to cover flexibility needs at different time scales, including seasonal storage~\citep{Zhang2021Review,Schmidt2023Monetizing,Yang2024Role}. Therefore, the question of how to optimally operate these storage systems has never been more relevant.

Scheduling an energy storage system involves determining when to charge or discharge to maximize a desired objective.
As storage systems introduce links between periods and are operated continuously without a predetermined end date, the operational scheduling problem is naturally formulated as an infinite-horizon optimization problem, even though the physical asset has a finite lifetime due to degradation.
This infinite-horizon aspect is usually disregarded in the literature; however, there are some exceptions. Nascimento and Powell \cite{Nascimento2013Optimal} consider an infinite-horizon problem, but their approach is inexact as it is based on approximating the objective function of the problem. Van de Ven et al. \cite{vandeVen2013Optimal} do treat the problem as having an infinite horizon and show that the optimal policy has a structure with two energy threshold levels. When the state of energy is below the lower threshold, the system charges, and it discharges when the state of energy is above the upper threshold. However, their results are based on the assumption that costs are discounted, and they only manage to get an analytical expression for these thresholds in cases with perfect efficiencies. A similar result is obtained in \cite{Harsha2015Optimal}, under the assumption of positive prices.
Other papers mention the extension to an infinite horizon as a future research direction \cite{Finnah2022Integrated}.
Currently, an exact method that can solve the infinite-horizon problem while considering the possibility of inefficiencies, negative prices, and undiscounted future costs does not exist.

To ensure that decisions made for scheduling an energy storage system are future-aware, a common approach is to solve a rolling-horizon version of the problem. The problem is solved over a longer horizon, but only the decisions in the first periods are implemented. The rest of the horizon is only advisory. The window is then shifted to make the next decisions. The horizon corresponding to the decisions that must be implemented is called the \emph{decision horizon}, and the longer horizon over which the problem is solved is the \emph{planning horizon}. 
Throughout this paper, we follow the terminology commonly used in the operations research literature. Similar optimization problems are also studied within the model predictive control (MPC) community, where slightly different terminology is adopted. In that context, a rolling horizon is commonly referred to as a \textit{receding horizon}, the decision horizon as the \textit{control horizon}, and the planning horizon as the \textit{prediction horizon} \citep{Rawlings2017Model,Prat2024Optimal}. The underlying optimization framework is similar, and only the terminology differs.
The advantages of a rolling horizon are demonstrated in, e.g., \cite{Secomandi2015Merchant}. If the planning horizon is long enough, this approach may actually obtain an optimal solution of the infinite-horizon problem. In this case, the planning horizon is said to be a \emph{forecast horizon}. However, if the planning horizon is not a forecast horizon, then the best solution obtainable might be much worse than that of the infinite-horizon problem. For most energy storage system technologies, the investment costs are very high \citep{Harsha2015Optimal}. Therefore, operating them closer to the optimal infinite-horizon schedule is crucial to ensure profitability \citep{Weitzel2018Energy}.
On the other hand, if the planning horizon is too long, it can result in unnecessarily high forecasting costs, which may diminish their operational profits \citep{Bardhan2013Forecast}. 
It is therefore essential to study the question of forecast horizons for problems with energy storage systems and to be able to determine the \emph{minimum forecast horizon}.

Though the length of the planning horizon in the rolling-horizon approach is crucial, it is rarely discussed, and often simply given as a number, with no justification \citep{Weitzel2018Energy, Mercier2023Value, Diller2024Dynamic}. Identifying an appropriate planning horizon is described as a challenge and a research gap for problems with storage systems by  Sioshansi et al. \cite{Sioshansi2021Energy}. 
Cuisinier et al. \cite{Cuisinier2022New} propose to aggregate future time intervals to increase the length of the planning horizon at a minimum computational cost. This aggregation is again arbitrary.
Some authors do discuss the length of the planning horizon but use an empirical approach, by looking at the change in their solution when the horizon is varied, see, e.g., \cite{Houwing2017Least, Kannan2011Game, DelReal2014Combined, Mayhorn2017Multi}. 
The conclusions are only applicable to the respective case study considered and do not give guidance as to how to proceed in general.
A systematic approach to evaluating planning horizons is thus missing.

As a first step toward addressing this challenge, we focus on problems using deterministic forecasts, which are widely adopted in both practice and academic research due to their analytical tractability \citep{Weitzel2018Energy}. While stochastic methods offer a more realistic representation by explicitly modeling uncertainty, developing a theoretical analysis of planning horizons is significantly more tractable in a deterministic framework and serves as a necessary foundation for future extensions to stochastic settings.
An attempt at proposing such a theoretical analysis in a deterministic framework was made by  Cruise et al. \cite{Cruise2019Control}, who introduced an algorithm to determine planning and decision horizons for a finite-horizon energy storage scheduling problem.
The algorithm finds the respective lengths of the decision and planning horizons for which the solution to their full-horizon problem is obtained, based on properties of the primal and dual variables obtained from the optimality conditions of the problem.
This algorithm is applied in \cite{Cruise2018Impact} and further extended in \cite{Anjos2020Control}. 
However, their proofs are based on the assumption that the problem is finite and are no longer valid when this assumption is relaxed.
In other works, it has been found that if in the resulting schedule, the storage system hits both its minimum and maximum capacity limits (in any order), the horizon is a forecast horizon; see \cite{Jesudasan2014Scheduling, Flatley2016Optimal, Zhao2019Determining}. A similar result has long been known for the wheat trading problem, see for example \cite{Hartl1986Forward}.
In all these related works, however, the conditions for a planning horizon to be in a forecast horizon in storage scheduling problems are sufficient but not necessary, and therefore, cannot be used to determine the minimum forecast horizon.

The finite-horizon approximation of infinite-horizon dynamic problems is not a new area. The problem has been extensively studied in the literature, particularly within the fields of inventory management and lot sizing, see e.g., the survey by Chand et al. \cite{Chand2002Forecast}. However, the mathematical structure of energy storage scheduling differs fundamentally from these applications. Existing forecast-horizon methods have primarily been developed for systems in which the controller actively influences the state in a single direction, such as inventory replenishment or reservoir releases, while the opposite direction is governed by exogenous demand or inflow. In contrast, an energy storage system is characterized by bidirectional state control, where both charging and discharging are optimization variables. Combined with storage-specific features such as energy capacity, charging and discharging rate limits, efficiencies, leakage, and time-varying electricity prices, this requires a different characterization.
Nevertheless, some insights from these applications remain useful. Bhaskaran and Sethi \cite{Bhaskaran1987Decision} show that matching solutions over the decision horizon under two extreme terminal conditions is sufficient to identify a forecast horizon in production planning and warehousing problems, although these terminal conditions are expressed as bounds on the future value of current decisions and are difficult to evaluate.
A similar condition is included in an algorithm for detecting forecast horizons by Cheevaprawatdomrong and Smith \cite{Cheevaprawatdomrong2004Infinite},  
and applied to hydro-reservoir operation in Zhao et al.~\cite{Zhao2012Identifying}.
These works provide important inspiration for the results presented in this paper, which extend the forecast-horizon framework to the context of bidirectional state control, as represented by a broad class of energy storage scheduling problems.

The main contribution of this paper is a theoretical characterization of forecast horizons for energy storage scheduling problems. We introduce a \textit{necessary and sufficient condition} to determine whether a given planning horizon is a forecast horizon, providing, to the best of our knowledge, the first exact characterization for this class of problems. Building on this result, we show that forecast horizons do not necessarily exist and establish a bound on the suboptimality resulting from planning horizons that are too short. We further develop an algorithm that determines the minimum forecast horizon \textit{a posteriori}, providing a systematic framework to inform the selection of planning horizons in future implementations. Although the methodology is introduced for a price-taking arbitrage problem, we show that the theoretical results apply to a broad class of energy storage scheduling problems, including market-clearing formulations, network-constrained optimization with DC-OPF, and behind-the-meter scheduling with local demand and renewable generation, under a wide range of electricity pricing schemes. Finally, we demonstrate how this framework can be applied in practice through a numerical study, illustrating how storage characteristics influence the minimum forecast horizon and quantifying the consequences of using planning horizons of arbitrary length.

The rest of this paper is organized as follows. In Sect.~\ref{sec:setup}, we introduce the storage scheduling problem and formally define the different horizons. In Sect.~\ref{sec:prop}, properties of forecast horizons are presented, including the simple condition to identify them, and a lower bound on the forecast horizon length. An algorithm to retrieve the minimum forecast horizon is also described. 
In Sect.~\ref{sec:ex}, the previous results are applied in various case studies. Section~\ref{sec:prac} summarizes the theoretical results in the form of practical guidelines for their application. Readers interested primarily in the practical application of the proposed method may wish to begin with Sect.~\ref{sec:prac} and refer to the preceding sections as needed for the underlying theoretical justification. Section~\ref{sec:ccl} concludes the paper. The complete proofs of the main results, including intermediary results and their proofs, are given in the appendices.
\section{Problem set-up and definitions}\label{sec:setup}

\subsection{Model and assumptions}\label{sec:model}

We consider the problem of scheduling a price-taker energy storage system, using energy price forecasts, and taking into account the operating constraints. We use a minimum representation of the storage system, including leakage and charging and discharging efficiencies, as well as limits for the state of energy and for charging and discharging. We consider that the storage system can be continuously dispatched between the minimum and maximum capacities.

The model when scheduling the operation of the storage over the time periods $\mathcal{T}=\{1,2,...,T\}$ is
\begin{subequations} \label{pb:opt}
\allowdisplaybreaks
\begin{align}
    \label{eq:obj} \max_{\mathbf{p\up{D}, p\up{C}, s}} \quad & \sum_{t \in \mathcal{T}} \Delta t \, C_t (p_t\up{D}-p_t\up{C}) \\
    \label{eq:init} \text{s.t.} \quad & s_1 = \rho S\up{init} +  \Delta t \left(\eta\up{C} p_1\up{C} - \frac{1}{\eta\up{D}} p_1\up{D}\right) \, , \\
    \label{eq:update} & s_t = \rho s_{t-1} +  \Delta t \left(\eta\up{C} p_t\up{C} - \frac{1}{\eta\up{D}} p_t\up{D}\right) \, ,& \, \forall t \in \mathcal{T} \setminus \{1\} \, , \\
    \label{eq:bounds_s} & \underline{S} \leq s_t \leq \overline{S} \, ,& \forall t \in \mathcal{T} \, , \\
    \label{eq:bounds_pc} & 0 \leq p_t\up{C} \leq \overline{P}\up{C} \, ,& \forall t \in \mathcal{T} \, , \\
    \label{eq:bounds_pd} & 0 \leq p_t\up{D} \leq \overline{P}\up{D} \, ,& \forall t \in \mathcal{T} \, \\
    \label{eq:cc} &  p_t\up{C} p_t\up{D} = 0 \, ,& \forall t \in \mathcal{T} .
\end{align}
\end{subequations}
The decision variables are the state of energy of the storage system, $s_t$, and the power charged and discharged during each time period of duration $\Delta t$, $p_t\up{C}$, and $p_t\up{D}$. These variables are bounded in \eqref{eq:bounds_s}-\eqref{eq:bounds_pd}. Parameters $\underline{S}$ and $\overline{S}$ are the minimum and maximum state of energy, $\overline{P}\up{C}$ is the maximum rate of charge, and $\overline{P}\up{D}$ is the maximum rate of discharge.
The objective, given in \eqref{eq:obj}, is to maximize the profit from arbitrage, using the storage system. The energy price paid if charging, or received if discharging, is $C_t$. We gather these prices in the vector $\textbf{C} \in \mathcal{C}_\mathcal{T}$, where $\mathcal{C}_\mathcal{T}$ is the set of all possible price vectors over $\mathcal{T}$.
Constraints \eqref{eq:init} and \eqref{eq:update} update the state of energy for the first time period and for the rest of the time periods, respectively. The level of energy initially available in the storage is given by $S\up{init}$. We consider that the storage system has leakage, such that the energy left at the end of each time period is multiplied by a factor $\rho \in \, ]0,1]$ at the beginning of the next time period \citep{Jesudasan2014Scheduling,Cruise2019Control}. We also consider losses when charging and discharging using efficiencies $\eta\up{C} \in \, ]0,1]$ and $\eta\up{D} \in \, ]0,1]$, and denote the round-trip efficiency $\eta = \eta\up{C} \eta\up{D}$.
Finally, constraint \eqref{eq:cc} prevents the simultaneous charge and discharge of the storage. This can be linearized using binary variables, as explained by \cite{Pozo2022Linear}. 
We name this model $\textbf{S}(\mathcal{T},\textbf{C})$ and refer to its infinite-horizon variant as $\textbf{S}(\mathbb{N}^+,\textbf{C})$.

\subsection{Definition of the different horizons} \label{sec:def}

We formally define the different types of horizons and illustrate them with the help of Fig.~\ref{fig:horizons}.
We assume that the decision horizon, which corresponds to the horizon for which we commit the decisions, is fixed by the decision-maker.
In the following, $\mathcal{H}$ represents the set of time periods in the decision horizon, with $|\mathcal{H}| = H$ being its length. In Fig.~\ref{fig:soe1} and Fig.~\ref{fig:soe2}, $H=24$, so only the first 24 decisions are committed. The set of time periods in the planning horizon, which is the horizon used to solve the problem, is $\mathcal{T}$, and is assumed to have a length $|\mathcal{T}| = T \geq H$. In the illustrative examples, we have two different planning horizons, in particular, $T=36$ in Fig.~\ref{fig:price1} and Fig.~\ref{fig:soe1}, and $T=48$ in Fig.~\ref{fig:price2} and Fig.~\ref{fig:soe2}.

\begin{figure*}[ht]
  \centering

  \begin{subfigure}{.45\linewidth}
    \centering
    \resizebox{\linewidth}{!}{\pgfplotsset{
    myplotstyle/.style={
    legend style={font=\small},
    legend cell align=left,
    legend pos=outer north east,
    tick label style={font=\small},
    ylabel style={align=center},
    xlabel style={align=center},
    axis x line*=bottom,
    axis y line*=left,
    scaled ticks=false,
    every axis plot/.append style={thick},
    },
}

\begin{tikzpicture}
\begin{axis}[
    myplotstyle,
    width=10cm,height=6cm,
    ymin=0, ymax=45, xmin=0, xmax=61,
    xtick = {0,6,12,18,24,30,36,42,48,54,60},
    xlabel={Time period $t$},
    ylabel={Price (€/kWh)},
    legend entries={\\$\widehat{\textbf{C}}^1$\\$\textbf{C}\up{1}$\\$\textbf{C}\up{2}$\\},
]
    \addplot[thick, samples=50, name path=daysplit2, black!75, dashed] coordinates {(36,0)(36,100)};
\addplot+[mark=none, color=black] table[x=t, y=Cf1, col sep=comma] {figures/horizons-plt-3.csv};
\addplot+[mark=none, color=blue, dotted] table[x=t, y=C1, col sep=comma] {figures/horizons-plt-3.csv};
\addplot+[mark=none, color=red, dashed] table[x=t, y=C2, col sep=comma] {figures/horizons-plt-3.csv};
    
        \path (axis cs:0.1, 0) coordinate (origin);
        \path (axis cs:1, 47) coordinate (D1_1);
        \path (axis cs:24, 47) coordinate (D1_2);
        \path (axis cs:1, 47) coordinate (D2_1);
        \path (axis cs:36, 47) coordinate (D2_2);
\end{axis}

\begin{scope}[decoration={calligraphic brace, amplitude=6pt}]
    \draw[thick,decorate] (D2_1) -- (D2_2) node[midway,above=1ex]{$\mathcal{T}$};
\end{scope}

\end{tikzpicture}}
    \caption{Price forecast and two scenarios for $T=36$}
    \label{fig:price1}
  \end{subfigure}
  \hfill
  \begin{subfigure}{.45\linewidth}
    \centering
    \resizebox{\linewidth}{!}{\pgfplotsset{
    myplotstyle/.style={
    legend style={font=\small},
    legend cell align=left,
    legend pos=outer north east,
    tick label style={font=\small},
    ylabel style={align=center},
    xlabel style={align=center},
    axis x line*=bottom,
    axis y line*=left,
    scaled ticks=false,
    every axis plot/.append style={thick},
    },
}

\begin{tikzpicture}
\begin{axis}[
    myplotstyle,
    width=10cm,height=6cm,
    ymin=0, ymax=100, xmin=0, xmax=61,
    xtick = {0,6,12,18,24,30,36,42,48,54,60},
    xlabel={Time period $t$},
    ylabel={State of energy (kWh)},
    legend entries={\\\\$\widehat{\textbf{C}}^1$\\$\textbf{C}\up{1}$\\$\textbf{C}\up{2}$\\},
]
    \addplot[thick, samples=50, name path=daysplit, black!75, dashed] coordinates {(24,0)(24,100)};
    \addplot[thick, samples=50, name path=daysplit2, black!75, dashed] coordinates {(36,0)(36,100)};
\addplot+[mark=none, color=black] table[x=t, y=sf1, col sep=comma] {figures/horizons-plt-2.csv};
\addplot+[mark=none, color=blue, dotted] table[x=t, y=s1, col sep=comma] {figures/horizons-plt-2.csv};
\addplot+[mark=none, color=red, dashed] table[x=t, y=s2, col sep=comma] {figures/horizons-plt-2.csv};
    \addplot[thick, name path=dh1, gray!25, opacity=0.5] coordinates {(0.05,0.05)(0.05,99.95)};
    \addplot[thick, name path=dh2, gray!25, opacity=0.5] coordinates {(23.95,0.05)(23.95,99.95)};
    \addplot[color=gray!25, opacity=0.5] fill between[of=dh1 and dh2];
    
        \path (axis cs:0.1, 0) coordinate (origin);
        \path (axis cs:1, 102) coordinate (D1_1);
        \path (axis cs:24, 102) coordinate (D1_2);
        \path (axis cs:1, 113) coordinate (D2_1);
        \path (axis cs:36, 113) coordinate (D2_2);
\end{axis}

\begin{scope}[decoration={calligraphic brace, amplitude=6pt}]
    \draw[thick,decorate] (D1_1) -- (D1_2) node[midway,above=1ex]{$\mathcal{H}$};
    \draw[thick,decorate] (D2_1) -- (D2_2) node[midway,above=1ex]{$\mathcal{T}$};
\end{scope}

\end{tikzpicture}}
    \caption{Resulting state of energy for $T=36$}
    \label{fig:soe1}
  \end{subfigure}

  \vspace{0.5em}

  \begin{subfigure}{.45\linewidth}
    \centering
    \resizebox{\linewidth}{!}{\pgfplotsset{
    myplotstyle/.style={
    legend style={font=\small},
    legend cell align=left,
    legend pos=outer north east,
    tick label style={font=\small},
    ylabel style={align=center},
    xlabel style={align=center},
    axis x line*=bottom,
    axis y line*=left,
    scaled ticks=false,
    every axis plot/.append style={thick},
    },
}

\begin{tikzpicture}
\begin{axis}[
    myplotstyle,
    width=10cm,height=6cm,
    ymin=0, ymax=45, xmin=0, xmax=61,
    xtick = {0,6,12,18,24,30,36,42,48,54,60},
    xlabel={Time period $t$},
    ylabel={Price (€/kWh)},
    legend entries={\\$\widehat{\textbf{C}}^2$\\$\textbf{C}\up{3}$\\$\textbf{C}\up{4}$\\},
]
    \addplot[thick, samples=50, name path=daysplit2, black!75, dashed] coordinates {(48,0)(48,100)};
\addplot+[mark=none, color=black] table[x=t, y=Cf2, col sep=comma] {figures/horizons-plt-3.csv};
\addplot+[mark=none, color=blue, dotted] table[x=t, y=C3, col sep=comma] {figures/horizons-plt-3.csv};
\addplot+[mark=none, color=red, dashed] table[x=t, y=C4, col sep=comma] {figures/horizons-plt-3.csv};

        \path (axis cs:0.1, 0) coordinate (origin);
        \path (axis cs:1, 47) coordinate (D1_1);
        \path (axis cs:24, 47) coordinate (D1_2);
        \path (axis cs:1, 47) coordinate (D2_1);
        \path (axis cs:48, 47) coordinate (D2_2);
\end{axis}

\begin{scope}[decoration={calligraphic brace, amplitude=6pt}]
    \draw[thick,decorate] (D2_1) -- (D2_2) node[midway,above=1ex]{$\mathcal{T}$};
\end{scope}

\end{tikzpicture}}
    \caption{Price forecast and two scenarios for $T=48$}
    \label{fig:price2}
  \end{subfigure}
  \hfill
  \begin{subfigure}{.45\linewidth}
    \centering
    \resizebox{\linewidth}{!}{\pgfplotsset{
    myplotstyle/.style={
    legend style={font=\small},
    legend cell align=left,
    legend pos=outer north east,
    tick label style={font=\small},
    ylabel style={align=center},
    xlabel style={align=center},
    axis x line*=bottom,
    axis y line*=left,
    scaled ticks=false,
    every axis plot/.append style={thick},
    },
}

\begin{tikzpicture}
\begin{axis}[
    myplotstyle,
    width=10cm,height=6cm,
    ymin=0, ymax=100, xmin=0, xmax=61,
    xtick = {0,6,12,18,24,30,36,42,48,54,60},
    xlabel={Time period $t$},
    ylabel={State of energy (kWh)},
    legend entries={\\\\$\widehat{\textbf{C}}^2$\\$\textbf{C}\up{3}$\\$\textbf{C}\up{4}$\\},
]
    \addplot[thick, samples=50, name path=daysplit, black!75, dashed] coordinates {(24,0)(24,100)};
    \addplot[thick, samples=50, name path=daysplit2, black!75, dashed] coordinates {(48,0)(48,100)};
\addplot+[mark=none, color=black] table[x=t, y=sf2, col sep=comma] {figures/horizons-plt-2.csv};
\addplot+[mark=none, color=blue, dotted] table[x=t, y=s3, col sep=comma] {figures/horizons-plt-2.csv};
\addplot+[mark=none, color=red, dashed] table[x=t, y=s4, col sep=comma] {figures/horizons-plt-2.csv};
    \addplot[thick, name path=dh1, gray!25, opacity=0.5] coordinates {(0.05,0.05)(0.05,99.95)};
    \addplot[thick, name path=dh2, gray!25, opacity=0.5] coordinates {(23.95,0.05)(23.95,99.95)};
    \addplot[color=gray!25, opacity=0.5] fill between[of=dh1 and dh2];
    
        \path (axis cs:0.1, 0) coordinate (origin);
        \path (axis cs:1, 102) coordinate (D1_1);
        \path (axis cs:24, 102) coordinate (D1_2);
        \path (axis cs:1, 113) coordinate (D2_1);
        \path (axis cs:48, 113) coordinate (D2_2);
\end{axis}

\begin{scope}[decoration={calligraphic brace, amplitude=6pt}]
    \draw[thick,decorate] (D1_1) -- (D1_2) node[midway,above=1ex]{$\mathcal{H}$};
    \draw[thick,decorate] (D2_1) -- (D2_2) node[midway,above=1ex]{$\mathcal{T}$};
\end{scope}

\end{tikzpicture}}
    \caption{Resulting state of energy for $T=48$}
    \label{fig:soe2}
  \end{subfigure}

  \caption{Illustrative example of the different horizons, for a decision horizon $H=24$ and for two different planning horizons}
  \label{fig:horizons}
\end{figure*}

The planning horizon is a forecast horizon if it results in an optimal solution of the infinite-horizon problem \citep{Chand2002Forecast,Ghate2011Infinite}.
We introduce the notation $\mathcal{X}_\mathcal{H}(\textbf{P})$ to represent the set of optimal solutions of problem $\textbf{P}$ over the decision horizon $\mathcal{H}$. Finally, a subscript on $\textbf{C}$ indicates the restriction of the vector to the given set of time periods, for example, $\textbf{C}_\mathcal{T}$.

\begin{definition} [Forecast horizon] \onehalfspacing
    Consider a decision horizon $\mathcal{H}$ and a planning horizon $\mathcal{T}$. 
    The planning horizon $\mathcal{T}$ is a \textbf{forecast horizon} for the price forecast $\widehat{\textbf{C}}$ if and only if an optimal solution over $\mathcal{H}$ is also optimal over $\mathcal{H}$ in the infinite-horizon problem, regardless of the value of the prices after the end of the planning horizon:
    \begin{align*}
        \mathcal{X}_\mathcal{H}\left(\textbf{S}(\mathcal{T},\widehat{\textbf{C}})\right) \subseteq \mathcal{X}_\mathcal{H}\left(\textbf{S}(\mathbb{N}^+,\textbf{C})\right), \, \forall \textbf{C} \in \mathcal{C}_{\mathbb{N}^+}, \, \textbf{C}_\mathcal{T} = \widehat{\textbf{C}}.
    \end{align*}
\end{definition}
We also use the expression \emph{long enough} to qualify a planning horizon that is a forecast horizon.

In Fig.~\ref{fig:horizons}, we can see that $\mathcal{T}=36$ is not a forecast horizon for the forecast $\widehat{\textbf{C}}^1$ since when we solve the problem for two different price vectors $\textbf{C}\up{1}$ and $\textbf{C}\up{2}$, which are such that $\textbf{C}\up{1}_\mathcal{T}=\textbf{C}\up{2}_\mathcal{T}=\widehat{\textbf{C}}^1$, as seen in Fig.~\ref{fig:price1}, the resulting state of energy, shown in Fig.~\ref{fig:soe1}, is not the same over the decision horizon. However, $\mathcal{T}=48$ might be a forecast horizon for the price vector $\widehat{\textbf{C}}^2$, since for two different price vectors $\textbf{C}\up{3}$ and $\textbf{C}\up{4}$, which are such that $\textbf{C}\up{3}_\mathcal{T}=\textbf{C}\up{4}_\mathcal{T}=\widehat{\textbf{C}}^2$, as seen in Fig.~\ref{fig:price2}, the solutions over $\mathcal{H}$ coincide, as shown in Fig.~\ref{fig:soe2}. However, to claim that $\mathcal{T}=48$ is a forecast horizon for $\widehat{\textbf{C}}^2$, this property must hold \emph{for all} $\textbf{C} \in \mathcal{C}_{\mathbb{N}^+}, \, \textbf{C}_\mathcal{T} = \widehat{\textbf{C}}^2$.

Note that if $\mathcal{T}$ is a forecast horizon, any longer planning horizon is also a forecast horizon. We are interested in finding the shortest of these, which we call the \emph{minimum forecast horizon}. 
\begin{definition} [Minimum forecast horizon] \onehalfspacing
    The planning horizon $\mathcal{T}$ is the \textbf{minimum forecast horizon} for the price forecast $\widehat{\textbf{C}}$ if it is a forecast horizon and if all shorter planning horizons are not forecast horizons.
\end{definition}

Next, we show some properties of (minimum) forecast horizons for the storage scheduling problem.
\section{Characterization of forecast horizons}\label{sec:prop}

\subsection{Condition for the identification of forecast horizons} \label{sec:criteria}

It is possible to evaluate the minimum reachable level at the end of a given planning horizon $T$, $\underline{S}_T$, by considering that, from the initial level, the maximum quantity is discharged over the full planning horizon, or until the minimum state of energy is reached, also accounting for losses when discharging and for leakage:
\begin{equation} \label{eq:sT_min}
    \underline{S}_T = \max \left \{\underline{S}, \rho^{T} S\up{init} - \Delta t \sum_{t=0}^{T-1} \rho^t \frac{1}{\eta\up{D}} \overline{P}\up{D}  \right \}.
\end{equation}
Similarly, the maximum reachable level at the end of the planning horizon, $\overline{S}_T$, can be obtained by considering that, from the initial level, the maximum quantity is charged over the full planning horizon, or until the maximum state of energy is reached, accounting for charging losses and for leakage:
\begin{equation} \label{eq:sT_max}
    \overline{S}_T = \min\left \{\overline{S}, \rho^{T} S\up{init} + \Delta t \sum_{t=0}^{T-1} \rho^t \eta\up{C} \overline{P}\up{C}  \right \} .
\end{equation}

We introduce $\textbf{F}(\mathcal{T}, \textbf{C}, S\up{end})$, which is a finite version of \eqref{pb:opt}, for which we have an additional constraint fixing the final state of energy in the storage to a given value $S\up{end}$:
\begin{equation}
    s_T = S\up{end}.
\end{equation}
We identify $\underline{\mathcal{X}}$ as the set of optimal solutions for problem $\textbf{F}(\mathcal{T}, \textbf{C}, \underline{S}_T)$. We use $\mathbf{\underline{x}}$ to identify an element of $\underline{\mathcal{X}}$, and $\underline{x}_t$ to identify the value of variable $x$ at $t$ for solution $\mathbf{\underline{x}}$.
Similarly, $\overline{\mathcal{X}}$ is the set of optimal solutions for problem $\textbf{F}(\mathcal{T}, \textbf{C}, \overline{S}_T)$, with elements $\mathbf{\overline{x}}$, and $\overline{x}_t$ is the value of variable $x$ at $t$ for solution $\mathbf{\overline{x}}$.
More generally, $\mathcal{X}^*$ is the set of optimal solutions for problem $\textbf{F}(\mathcal{T}, \textbf{C}, S\up{end})$, with $\underline{S}_T \leq S\up{end} \leq \overline{S}_T$, with elements $\mathbf{{x}}^*$, and ${x}_t^*$ is the value of variable $x$ at $t$ for solution $\mathbf{x}^*$.

Our result, formalized in Theorem~\ref{theorem1}, states that if, when solving $\textbf{F}(\mathcal{T}, \textbf{C}, \underline{S}_T)$ and $\textbf{F}(\mathcal{T}, \textbf{C}, \overline{S}_T)$, the state of energy at the end of the decision horizon is the same for both problems, then the planning horizon $\mathcal{T}$ is a forecast horizon. The proof is based on different intermediary results and is available in Appendix~\ref{app:proof1}.

\begin{theorem} \label{theorem1} \onehalfspacing
    The planning horizon $\mathcal{T}$ is a forecast horizon if and only if $\exists \, \underline{\mathbf{x}} \in \underline{\mathcal{X}}$ and $\exists \, \overline{\mathbf{x}} \in \overline{\mathcal{X}}$ such that $\underline{s}_H = \overline{s}_H$.
\end{theorem}

This result covers the case of multiple optimal solutions: it is enough to find a solution to each problem such that $\underline{s}_H = \overline{s}_H$.
In Sect.~\ref{sec:algo}, we describe an algorithm for finding the minimum forecast horizon that evaluates the condition in the case of multiple optimal solutions.

If there are no solutions such that $\underline{s}_H = \overline{s}_H$, the planning horizon is not a forecast horizon. However, a small gap between these two values indicates that the error made in this case will be limited.
We next show how the gap between $\underline{s}_H$ and $\overline{s}_H$ is related to suboptimality.

\subsection{An upper bound on suboptimality} \label{sec:subopt}

Evaluating suboptimality with respect to the solution to the infinite-horizon problem is far from being a simple task. 
Indeed, making poor decisions in the current decision horizon could still have an impact long after those decisions were made, and errors in future decision horizons might further amplify the overall suboptimality. 
Therefore, we need to finitely constrain what we understand by suboptimality to be able to evaluate it.
Accordingly, here we assess suboptimality by assuming that an error is only made in the current decision horizon and that all future decisions are optimal (i.e., all future planning horizons are forecast horizons).

We further assume that future prices, outside of the decision horizon, are bounded from below by $\underline{C} \leq 0$ and from above by $\overline{C} \geq 0$. With these assumptions and conditions in place, the following result applies. The proof is given in Appendix~\ref{app:proof3}.

\begin{proposition} \label{prop2} \onehalfspacing
    For a given level $s_H$ of the storage at the end of the decision horizon such that $\underline{s}_H \leq s_H \leq \overline{s}_H$, the deviation from the infinite-horizon objective $\Delta Z$ is upper-bounded as follows:
    \begin{equation} \label{eq:subopt1}
        \Delta Z \leq Z\up{opt,DH} - Z\up{DH} + \max \left \{- \underline{C} \frac{1}{\eta\up{C}} (s_H - \underline{s}_H),  \overline{C} \eta\up{D} (\overline{s}_H - s_H) \right \},
    \end{equation}
    where $Z\up{DH}$ is the value of the objective function over the decision horizon corresponding to $s_H$, and $Z\up{opt,DH}$ is the objective value achieved when solving the scheduling problem over the decision horizon only and considering that the final storage at the end of the decision horizon is a variable bounded between  $\underline{s}_H$ and $ \overline{s}_H$.
\end{proposition}

Therefore, the lower the gap between $\overline{s}_H$ and $\underline{s}_H$, the tighter the upper bound on the suboptimality gap with respect to the infinite-horizon solution.
We can also choose $s_H$ so that this upper bound on suboptimality is minimized.

An important advantage of the proposed bound is that it does not rely on any assumptions regarding the future price process and therefore remains valid regardless of the evolution of future prices. As a consequence of this generality, however, the bound is conservative. Tighter bounds could likely be obtained by exploiting additional information on future prices, such as probability distributions, temporal correlations, or other structural properties of the price process. The bound could also be improved by accounting for the effect of leakage, which further reduces the influence of distant future periods on the current decision. Developing such refinements is an interesting direction for future work. Section~\ref{sec:res_subopt} further discusses the conservativeness of the proposed bound based on the numerical experiments.

\subsection{Existence of forecast horizons} \label{sec:existence}

Prior to determining forecast horizons, it is useful to establish whether they exist or not. In fact, the non-existence of forecast horizons under certain conditions is well-known from the literature on infinite-horizon optimization problems \citep{Chand2002Forecast, Cheevaprawatdomrong2007Solution, Lortz2015Solvability}.
By way of an example, we demonstrate that the existence of a forecast horizon is not guaranteed in energy storage scheduling problems.

Consider a storage system with charging and discharging rates such that the storage system can be charged or discharged completely in one time period, with $\eta < 1$, and with an initial level such that $\underline{S}<S\up{init}<\overline{S}$.
The decision horizon is one time period.
The price sequence is such that $C_1 > 0$ and $\eta C_1 < C_t < C_1$, $\forall t > 1$.
The resulting states of energy from solving $\textbf{F}(\mathcal{T}, \textbf{C}, \underline{S}_T)$ and $\textbf{F}(\mathcal{T}, \textbf{C}, \overline{S}_T)$ are shown in Fig.~\ref{fig:existence}.

\begin{figure}[thb] 
    \centering
    {\resizebox{.6\linewidth}{!}{

\begin{tikzpicture}
    \begin{axis}[
        xlabel={Time period},
        ylabel={State of energy (kWh)},
        xmin=0, xmax=10,
        ymin=0, ymax=10,
        xtick={0,1,2,3,4,5,6,7,8,9,10},
        xticklabels={$0$,$1$,$2$,$3$,$4$,$5$,$6$,$7$,$\cdots$,$T-1$,$T$},
        ytick={0,5,10},
        yticklabels={$\underline{S}$,$S\up{init}$,$\overline{S}$},
        legend style={at={(1.05,0.5)}, anchor=west, draw=none, 
                      font=\small},
        ]
        \addplot[
            color=red, style=densely dashdotted,
            line width=1.5pt,
            ] coordinates {
            (0,5) (1,5) (9,5) (10,10)
        };
        \addplot[
            color=blue, style=densely dashed,
            line width=1.5pt,
            ] coordinates {
            (0,5) (1,0) (10,0)
        };

        \addlegendentry{$\textbf{F}(\mathcal{T}, \textbf{C}, \overline{S}_T)$}
        \addlegendentry{$\textbf{F}(\mathcal{T}, \textbf{C}, \underline{S}_T)$}
    \end{axis}
\end{tikzpicture}
}}
    \caption{Example of non-existence of a forecast horizon} \label{fig:existence}
\end{figure}

According to Theorem~\ref{theorem1}, to have a forecast horizon, the respective states of energy must be equal in the first time period.
For the red curve to match the blue curve in the first period, it needs to be profitable to discharge in the first period, at price $C_1$, and later charge at a lower price $C_t$.
However, we also have to account for inefficiencies: $\Delta s = s_{t}-s_{t-1} = \Delta t \left(\eta\up{C} p_t\up{C} - \frac{1}{\eta\up{D}} p_t\up{D}\right)$ (without leakage). So if $\Delta s \leq 0$, the corresponding discharge is $p_t\up{D} = \frac{- \Delta s \eta\up{D}}{\Delta t}$. And if $\Delta s \geq 0$, the corresponding charge is $p_t\up{C} = \frac{\Delta s}{\eta\up{C} \Delta t}$.
The change in the objective function $\Delta Z$ associated with the discharge of $\Delta s$ in the first period and the charge of $\Delta s$ at a later $t$ is $\Delta Z = C_1 \Delta s \eta\up{D} - C_t \frac{\Delta s}{\eta\up{C}}$, so it is only profitable to discharge in the first period if there is $t>1$ such that $C_t \leq \eta C_1$. 
\\
On the other hand, for the blue curve to match the red curve in the first time period, it needs to be profitable to discharge less in the first time period and discharge more later, which is the case if there is $t>1$ such that $C_t \geq C_1$. 
\\
Therefore, if $\eta C_1 < C_t < C_1$ for all $t > 1$, the two decisions in the first period will never match, which means that there exists no forecast horizon.

The analysis above suggests that the existence of a forecast horizon depends on the evolution of the price, also accounting for efficiency.
The lower the efficiency, the larger the variation of prices needed to make a charging-discharging action profitable.

The fact that the existence of forecast horizons is not guaranteed can limit the applicability of algorithms for discovering forecast horizons. Algorithms based on iteratively increasing the length of the planning horizon may fail to converge if there is no forecast horizon at all. This can be avoided by setting a maximum to the length of the planning horizon, as in \cite{Cruise2019Control}; however, there is no guarantee that a forecast horizon is found by the time the maximum is reached.

\subsection{Necessary condition for forecast horizons}

For the problem considered, it is not possible to calculate the minimum forecast horizon without solving the problem, as it highly depends on the value of the prices. However, we can obtain a necessary condition for forecast horizons prior to solving the problem, as given by Proposition~\ref{prop1}. The proof is available in Appendix~\ref{app:proof2}.

\begin{proposition}\label{prop1} \onehalfspacing
For a planning horizon $\mathcal{T}$, of length $T$, to be a forecast horizon, it must satisfy
\begin{align} \label{eq:cond2}
    \nonumber \min  & \left\{ \overline{S}  - \underline{S} - \left ( \sum_{t=0}^{T-H-1} \rho^t \right ) \left(\Delta t \, \eta\up{C} \overline{P}\up{C} + \frac{\Delta t}{\eta\up{D}} \overline{P}\up{D} \right), \right. \\
    \nonumber & \left. \rho^{T} S\up{init} - \underline{S} + \Delta t \, \eta\up{C}  \overline{P}\up{C} \left ( \sum_{t=T-H}^{T-1} \rho^t \right ) - \frac{\Delta t}{\eta\up{D}} \overline{P}\up{D} \left ( \sum_{t=0}^{T-H-1} \rho^t \right ), \, \right.\\
    & \left. \overline{S} - \rho^{T} S\up{init} -  \Delta t \, \eta\up{C}  \overline{P}\up{C} \left ( \sum_{t=0}^{T-H-1} \rho^t \right ) + \frac{\Delta t}{\eta\up{D}} \overline{P}\up{D} \left ( \sum_{t=T-H}^{T-1} \rho^t \right )  \right \} \leq 0 \, .
\end{align} 
\end{proposition}

If a planning horizon $\mathcal{T}$ does not satisfy \eqref{eq:cond2}, we can conclude that it is not a forecast horizon, without having to resort to Theorem~\ref{theorem1}.
Expression \eqref{eq:cond2} is similar to the one proposed in \cite{Ela2015Scheduling} that is used for determining the length of their planning horizon. In the case of storage system scheduling, \eqref{eq:cond2} is not sufficient to guarantee that the planning horizon is long enough, but can provide a starting point for an algorithm built to determine the minimum forecast horizon as presented in Sect.~\ref{sec:ex}. This starting point, referred to as $T\up{min}$, is the first $T \geq H$ such that \eqref{eq:cond2} is satisfied.

\subsection{Algorithm to determine the minimum forecast horizon} \label{sec:algo}

The results established above can be used to obtain the minimum forecast horizon of~\eqref{pb:opt}. We describe an algorithm to do so. 

\begin{algorithm}[htbp]
\caption{Minimum forecast horizon}
\label{alg}
\begin{algorithmic}[1]

\Require $T\up{max} \geq H$
\State Calculate $T\up{min}$ (Proposition~\ref{prop1})
\State $T \gets T\up{min}$
\State $gap \gets M$
\State $subopt \gets M$

\While{$T \leq T\up{max}$ \textbf{and} $gap > 0$}
    \State $\underline{s}_H \gets \text{Solve } \mathbf{F}(\mathcal{T}, \mathbf{C}, \underline{S}_T)$
    \State $\overline{s}_H \gets \text{Solve } \mathbf{F}(\mathcal{T}, \mathbf{C}, \overline{S}_T)$
    \State $gap \gets \overline{s}_H - \underline{s}_H$
    \If{$gap > 0$}
        \State $gap \gets \text{Solve } min\_gap$
    \EndIf
    \State $T \gets T + 1$
\EndWhile

\If{$gap > 0$}
    \State Calculate $subopt$
\Else
    \State $subopt \gets 0$
\EndIf

\State \Return $T$, $subopt$

\end{algorithmic}
\end{algorithm}

The main idea of Algorithm~\ref{alg} is to iteratively increase the length of the planning horizon until the condition given in Theorem~\ref{theorem1} is satisfied, which is the case when $gap = 0$, where $gap$ is initialized to $M>0$. This approach is similar to those of \cite{Garcia2000Solving} and \cite{Cheevaprawatdomrong2004Infinite}, but adapted to the energy storage scheduling problem.
The final gap is used to provide an upper bound on suboptimality, denoted by $subopt$, following, for example, the approach in Sect.~\ref{sec:subopt}.
The starting point for the planning horizon, indicated by $T\up{min}$, corresponds to the lower bound obtained with Proposition~\ref{prop1}.

The limit $T\up{max}$ should be chosen such that there is enough data available to cover it and a solution can be obtained in finite time.
If $T\up{max}$ is reached, a strictly positive value of $subopt$ indicates that, if it exists, the minimum forecast horizon is greater than $T\up{max}$. If this happens, and if it is possible to obtain data further into the future, one can repeat the process using a larger value of $T\up{max}$ and starting from the previous $T\up{max}$.

Note that the algorithm also includes a test for solution multiplicity, denoted as $min\_gap$. Even if $\underline{s}_H \neq \overline{s}_H$, there might be other optimal solutions such that this condition is satisfied. 
One way to check this is to solve a third problem that includes both scheduling problems, for $S\up{end}=\underline{S}_T$ and for $S\up{end}=\overline{S}_T$. The objective function minimizes the difference between $\underline{s}_H$ and $\overline{s}_H$, and constraints enforce that the profit for each scheduling problem is equal to the optimal value previously obtained when the problems were solved separately. Thus, this problem will identify if there exist other optimal solutions such that $\underline{s}_H=\overline{s}_H$.

To the best of our knowledge, this is the first algorithm capable of determining the minimum forecast horizon for energy storage scheduling problems. Knowing this minimum allows us to quantify the consequences of using a forecast horizon that is either too short or excessively long.
Importantly, the goal of analyzing the minimum forecast horizon is not to generate operational scheduling decisions, but to provide an a posteriori assessment of the forecasting requirements for a given problem. Consequently, computational efficiency is not the primary concern.
Yet, in the future, more efficient algorithms could be obtained, for example, by applying a bisection method or by using the solution of the previous iteration as a warm start.

\subsection{Validity of the results for more complex problems} \label{sec:complex}

The results presented in this section are valid for the model~\eqref{pb:opt}, but can be extended to problems including one or more storage systems, provided that they satisfy the following conditions.
First, the complementarity constraints must be relaxed.
Second, the objective function must be of the form
\begin{equation}
    \sum_{i \in \mathcal{I}} \sum_{t \in \mathcal{T}} C_{i,t} (p_{i,t}\up{D}-p_{i,t}\up{C}) + f(\mathbf{y}).
\end{equation}
The set $\mathcal{I}$ corresponds to the different storage systems. Parameters $C_{i,t} \in \mathbb{R}$ correspond to coefficients in the objective function for the difference $p_{i,t}\up{D}-p_{i,t}\up{C}$. The rest of the objective function relates to the rest of the variables, where $\mathbf{y}$ gathers all the variables that are not related to storage systems and $f(\mathbf{y})$ represents the contribution of these other variables to the objective function and is such that the problem is linear.

Third, the constraints involving storage system variables for different storage systems and some of the other variables in $\mathbf{y}$ must be of the form
\begin{equation}
    \sum_{i \in \mathcal{I}} \sum_{t \in \mathcal{T}} A_{i,t} (p_{i,t}\up{D}-p_{i,t}\up{C}) + h(\mathbf{y}) \leq 0,
\end{equation}
where parameters $A_{i,t} \in \mathbb{R}$ correspond to coefficients for the difference $p_{i,t}\up{D}-p_{i,t}\up{C}$ in these constraints, and $h(\mathbf{y})$ represents the contribution of the other variables to these constraints and is such that the problem is linear.
An example is constraints that give the balance between total production and total consumption.

Under these conditions, it is possible to decompose the problem, with one sub-problem for each storage system, for example, using Lagrangian relaxation or Dantzig-Wolfe decomposition. These sub-problems are of the same form as~\eqref{pb:opt}, therefore the previous results apply to each sub-problem. 
We can then use the fact that a planning horizon is a forecast horizon for the full problem if and only if it is a forecast horizon for each sub-problem.
Moreover, if $T\up{min}_i$ represents the minimum forecast horizon for the sub-problem corresponding to storage system $i$, the minimum forecast horizon of the full problem, $T\up{min}$, is $T\up{min} = \max_{i \in \mathcal{I}} T\up{min}_i$.
\section{Application}\label{sec:ex}
We apply Algorithm~\ref{alg} to different storage systems to evaluate the minimum forecast horizon for various storage characteristics and compare the results from solving the problem with a forecast horizon to those obtained with common approaches. In all these examples, we consider a decision horizon of 24 hours. The code for the case studies is available at \url{https://github.com/eleaprat/stg-horizons}.

\subsection{Test case}
The data for four storage systems is given in Table~\ref{tab:data_stg}. 
The first two have short durations of charge and discharge, which correspond to the minimum time needed to completely fill or empty the storage system, calculated as $\frac{\overline{S}-\underline{S}}{\eta\up{C}\overline{P}\up{C}}$ for the duration of charge and $\eta\up{D}\frac{\overline{S}-\underline{S}}{\overline{P}\up{D}}$ for the duration of discharge.
The durations of charge and discharge are five times as long in the last two cases.
The difference between the first two is that the second has lower efficiencies (eff.), while the durations of charge and discharge are kept. This results in distinct values for the maximum rates of charge and discharge.
The last storage system allows us to evaluate the impact of leakage. 

\begin{table}[hbt]
    \caption{Data for the storage systems of the test cases}
    \centering
    \begin{tabular}{lccccccc}
    \hline
    & \multicolumn{1}{c}{\begin{tabular}[c]{@{}c@{}}$\overline{P}\up{C}$ and  $\overline{P}\up{D}$ \\ (kW)\end{tabular}} & \multicolumn{1}{c}{{\begin{tabular}[c]{@{}c@{}}$\underline{S}$ \\ (kWh)\end{tabular}}} & \multicolumn{1}{c}{{\begin{tabular}[c]{@{}c@{}}$\overline{S}$ \\ (kWh)\end{tabular}}} & \multicolumn{1}{c}{$\eta\up{C}$ and $\eta\up{D}$} & \multicolumn{1}{c}{\begin{tabular}[c]{@{}c@{}}Duration of \\ charge (h)\end{tabular}} & \multicolumn{1}{c}{\begin{tabular}[c]{@{}c@{}}Duration of \\ discharge (h)\end{tabular}} & \multicolumn{1}{c}{$\rho$} \\ \hline
    Fast storage & 1  & 0 & 10 & 0.9 & 12 & 9  & 1.0 \\
    {\begin{tabular}[l]{@{}l@{}}Fast storage \\ (low eff.)\end{tabular}} & 1.5, 0.7 & 0 & 10 & 0.6 & 12 & 9  & 1.0 \\
    Slow storage & 1 & 0 & 50 & 0.9 & 56 & 45 & 1.0 \\
    {\begin{tabular}[l]{@{}l@{}}Slow storage \\ (leakage)\end{tabular}}  & 1 & 0 & 50 & 0.9 & 56 & 45 & 0.99 \\
    \hline
    \end{tabular}
    \label{tab:data_stg}
\end{table}

We compare the schedules over three months. The initial level is half capacity.
The prices are from the day-ahead market for Denmark (DK1) between September and December 2024. They are plotted in Fig.~\ref{fig:prices}. Prices from December are used to evaluate the minimum forecast horizon and storage schedules at the end of the third month.
In all the studies, we set $T\up{max}$ to be equal to the maximum horizon available in the data, i.e., until the end of December.

\begin{figure}[thb]
    \centering
    \includegraphics[width=0.9\linewidth]{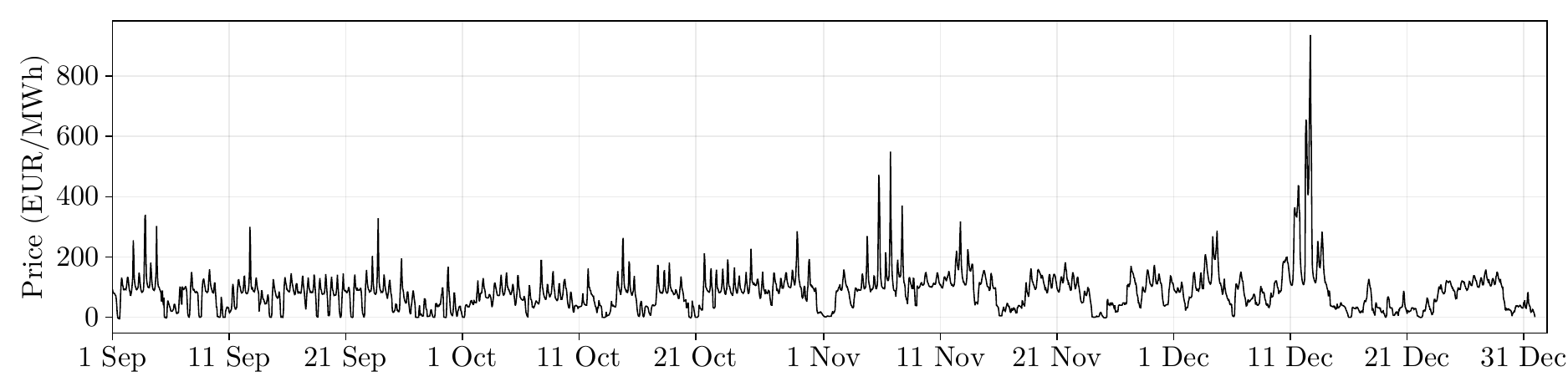}
    \caption{Prices used in the case study (day-ahead prices of DK1 between September and December 2024) plotted against time (hours)}
    \label{fig:prices}
\end{figure}

\subsection{Study of the minimum forecast horizon}

We study the minimum forecast horizon for the four case studies, following Algorithm~\ref{alg}. 
We obtain the plots in Fig.~\ref{fig:res_fcst}, which give the length of the minimum horizon in hours for each day of the three months considered and for each test system.
These plots also show the lower bound on the minimum forecast horizon obtained with Proposition~\ref{prop1} in black, and for comparison, the red line corresponds to 48 hours.

\begin{figure}[htbp]
    \centering

    \begin{subfigure}{0.49\linewidth}
        \centering
        \includegraphics[width=\linewidth]{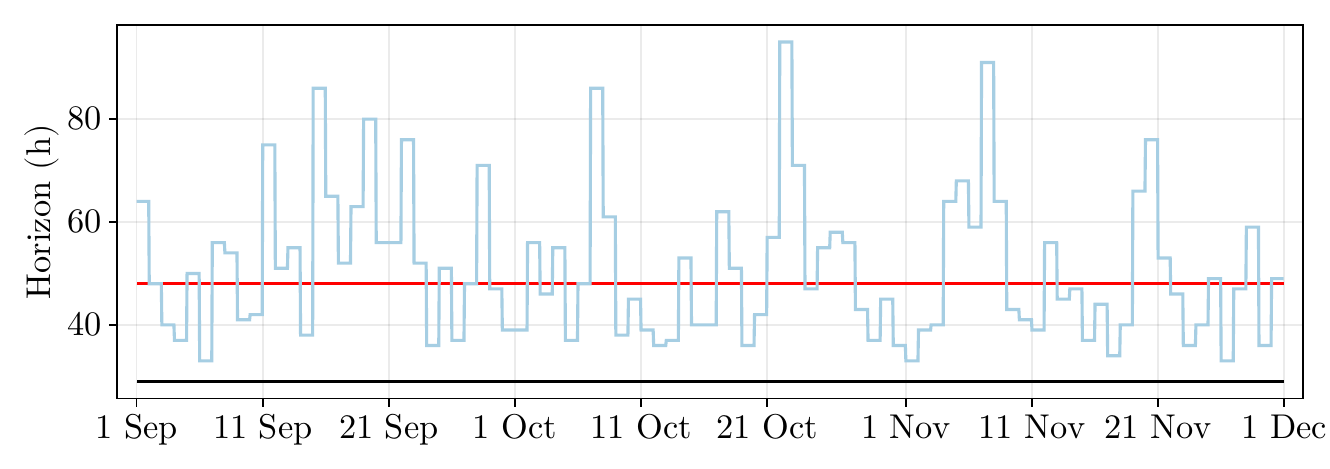}
        \caption{Fast storage}
        \label{fig:res_fcst_fast}
    \end{subfigure}
    \hfill
    \begin{subfigure}{0.49\linewidth}
        \centering
        \includegraphics[width=\linewidth]{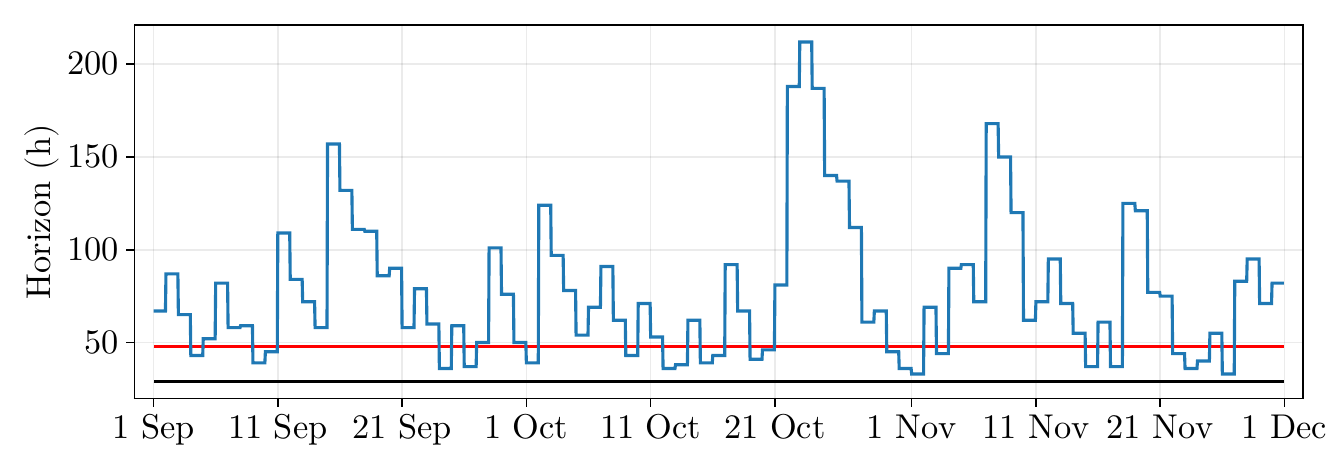}
        \caption{Fast storage (low eff.)}
        \label{fig:res_fcst_eff}
    \end{subfigure}

    \vspace{0.3em}

    \begin{subfigure}{0.49\linewidth}
        \centering
        \includegraphics[width=\linewidth]{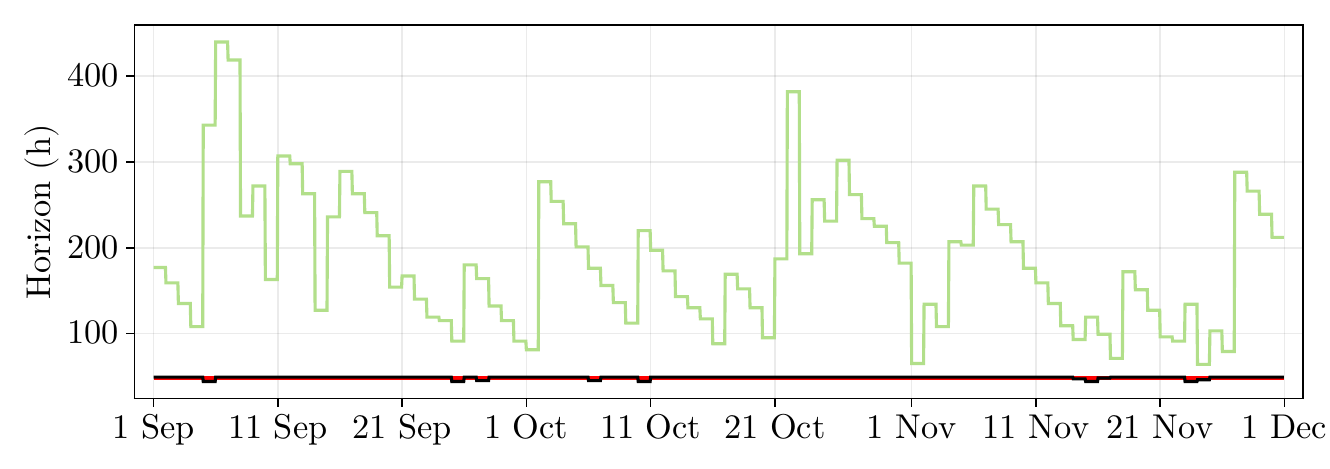}
        \caption{Slow storage}
        \label{fig:res_fcst_slow}
    \end{subfigure}
    \hfill
    \begin{subfigure}{0.49\linewidth}
        \centering
        \includegraphics[width=\linewidth]{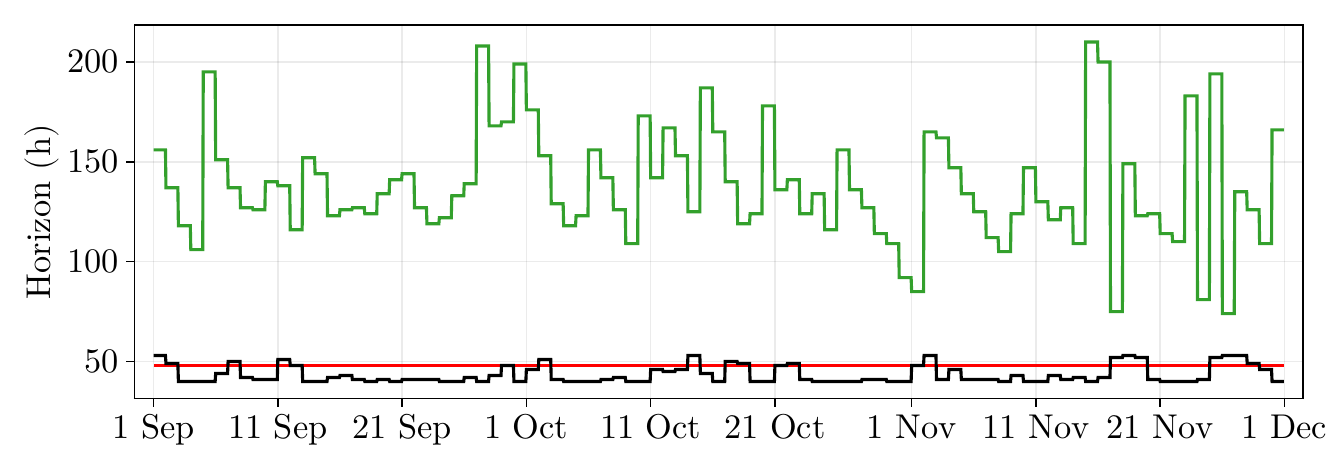}
        \caption{Slow storage (leakage)}
        \label{fig:res_fcst_leak}
    \end{subfigure}

    \caption{Minimum forecast horizon for the four storage systems, plotted against time (hours). The red line indicates 48 hours, and the black line indicates the lower bound on the minimum forecast horizon $T\up{min}$.}
    \label{fig:res_fcst}
\end{figure}

The minimum forecast horizon is often above 48 hours for the fast storage systems, and always well above 48 hours for the slow storage systems.
Overall, this casts doubt on why 48 hours is usually chosen as the length of the planning horizon. 
Also note that it varies significantly over time in all cases, which is an argument for not using fixed planning horizons.
For all cases, there is a large difference between the lower bound on the minimum forecast horizon $T\up{min}$ and the actual minimum forecast horizon. As $T\up{min}$ is calculated based on the technical characteristics of the storage system only, thus disregarding the prices and their evolution, this large difference emphasizes the importance prices have on the length of the minimum forecast horizon.
The results of Fig.~\ref{fig:res_fcst_fast} and Fig.~\ref{fig:res_fcst_eff} suggest that, for the same duration of charge and discharge, the efficiency of the storage has a significant impact on the length of the minimum forecast horizon, which is in line with the comments of Sect.~\ref{sec:existence}.
Comparing Fig.~\ref{fig:res_fcst_fast} and Fig.~\ref{fig:res_fcst_slow}, we see that the minimum forecast horizon increases for a storage system that takes longer to charge and discharge.
For the same charging and discharging duration, the minimum forecast horizon decreases when there is leakage, as we can observe by comparing Fig.~\ref{fig:res_fcst_slow} and Fig.~\ref{fig:res_fcst_leak}.

We show the execution time of Algorithm~\ref{alg} for each case study in Fig. \ref{fig:execution}. The maximum execution time is approximately 80 seconds, for the fast storage, but all typically require less than 10 seconds. Overall, these execution times remain reasonable, especially since Algorithm~\ref{alg} is meant to be run offline.
\begin{figure}[ht]
    \centering
    \includegraphics[width=0.6\linewidth]{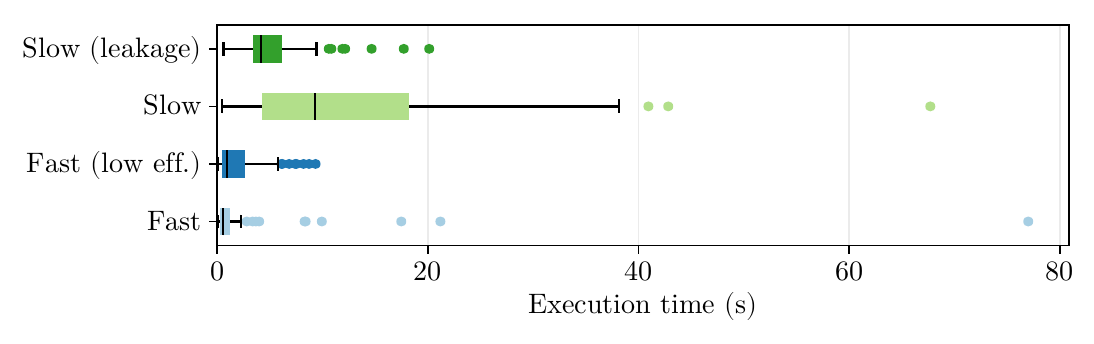}
    \caption{Overview of the execution time of Algorithm 1 for each case study}
    \label{fig:execution}
\end{figure}

In the case where these four storage systems were to be scheduled in the same problem under the same price, the plot in Fig.~\ref{fig:res_fcst_all} indicates which of the storage systems would determine the minimum forecast horizon for the complete problem.
The plotted curve corresponds to the maximum of the minimum forecast horizons of the four storage systems, and the background color shows which specific system sets this maximum.
We observe that the slow storage is often the one determining the minimum forecast horizon, but sometimes it is the slow storage with leakage, or the fast storage with low efficiency. On the contrary, the fast storage never sets the length of the minimum forecast horizon in this case study, and thus, could have been simply disregarded from the analysis.

\begin{figure}[ht]
    \centering
    \includegraphics[width=0.8\linewidth]{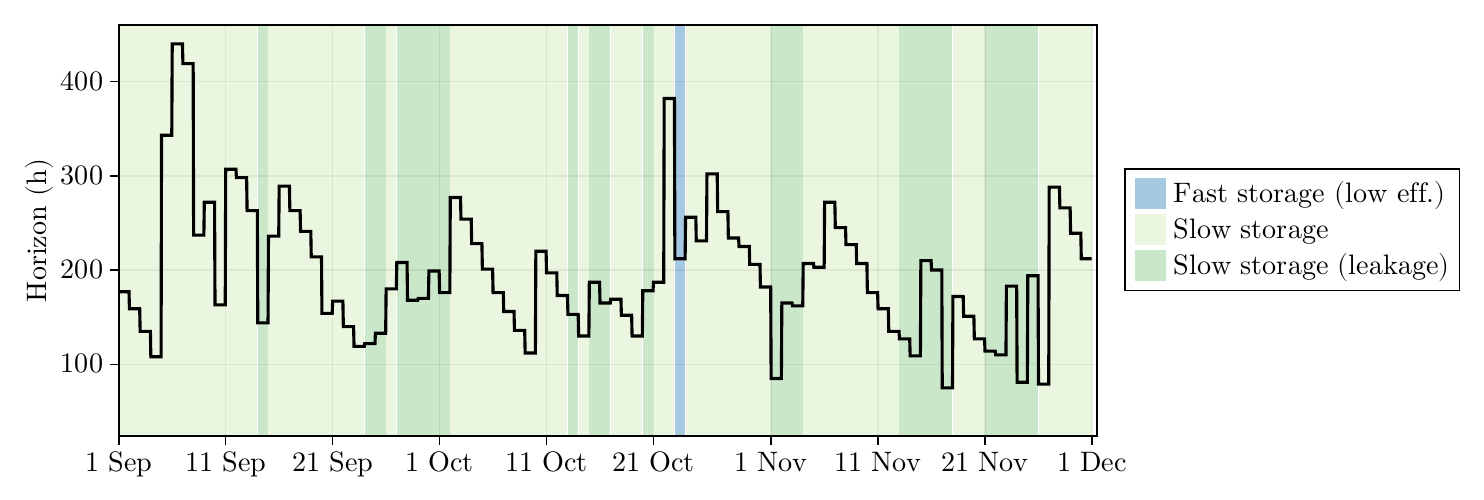}
    \caption{Minimum forecast horizon for the four storage systems, plotted against time (hours)}
    \label{fig:res_fcst_all}
\end{figure}

The minimum forecast horizons obtained here reach at most 410 hours, which remains substantially shorter than the remaining operational lifetime of even short-lived battery systems under intensive cycling. This illustrates that, for most of a storage system's lifetime, the finite end of operation is unlikely to influence the optimal scheduling decisions over the planning horizon. Only as the storage system approaches decommissioning is the proposed infinite-horizon framework no longer appropriate, as the finite remaining lifetime directly influences the optimal scheduling decisions.

\subsection{Influence of prices on the minimum forecast horizon}
To better understand how the price signal influences the minimum forecast horizon, we investigate the relationship between the minimum forecast horizon and several characteristics of the prices. 
For this analysis, we focus on the fast storage system.
For each day over the full year 2024, a set of metrics characterizing the electricity prices is computed over different time windows, starting from the first hour of that day and for 40, 60 or 80 hours in total.
The metrics considered are summarized in Table~\ref{tab:price_metrics}, along with how they are obtained and why they are interesting to look at.
\begin{table}[ht]
\caption{Details of the different metrics evaluated, where $N$ is the number of hourly observations in the considered time window, $\mu_C$ is the mean electricity price over the considered time window, $K$ is the number of monotonic price trends within the considered time window and $L_k$ is the duration (in hours) of the $k$-th monotonic price trend.}
\resizebox{\textwidth}{!}{%
\begin{tabular}{p{2.5cm} C{3.5cm} p{4.0cm} p{5.5cm}}
\hline
\textbf{Metric}        & \textbf{Formula}           & \textbf{Description}                                                       & \textbf{Motivation}                                                                                                                                \\ \hline
\makecell[l]{Standard\\deviation}    & $\sqrt{\frac{1}{N} \sum_{t=1}^{N} \left(C_t-\mu_{C} \right)^2}.$                  & Measures the overall variability of the price signal.                      & Reflects price fluctuations, which may constitute arbitrage opportunities.                                                                         \\
Relative price spread  & $\frac{\frac{1}{N}\sum_{t=1}^{N}\left|C_t-\mu_{C}\right|}{\frac{1}{N}\sum_{t=1}^{N}C_t}.$ & Mean absolute deviation from the mean price, normalized by the mean price. & Measures the typical relative magnitude of price fluctuations. Compared with the standard deviation, it is less sensitive to extreme price events. \\
Autocorrelation (24 h) & $\text{corr}(C_t,C_{t+24}).$  & Pearson correlation coefficient between the original price series and the same series shifted by 24 hours. & Strong daily periodicity may indicate more repetitive price patterns and fewer additional arbitrage opportunities beyond a daily cycle.            \\
Trend duration         & $\frac{1}{K}\sum_{k=1}^{K}L_k.$                   & Average duration of monotonic price trends.                                & Quantifies the persistence of price movements, which may influence how far ahead to consider.                                                      \\ \hline
\end{tabular}
}
\label{tab:price_metrics}
\end{table}

The minimum forecast horizon is plotted against each metric, a linear regression is fitted to evaluate the strength of the relationship, and the Pearson coefficient ($r$) and p-value ($p$) are retrieved. 
All these results are shown in Fig. \ref{fig:scatter}.
\begin{figure}[ht]
\centering
\setlength{\tabcolsep}{2pt}
\resizebox{\textwidth}{!}{%
\begin{tabular}{M{1cm}cccc}

&
\textbf{Standard deviation} &
\textbf{Trend duration} &
\textbf{Autocorrelation} &
\textbf{Relative price spread}
\\[2mm]

\textbf{40 h} &
\includegraphics[valign=c,width=.22\textwidth]{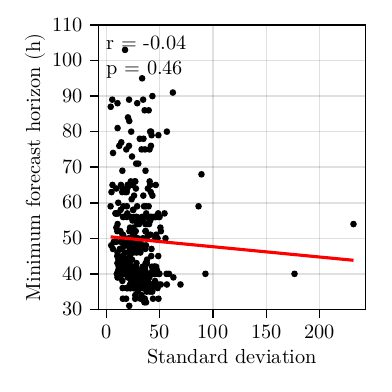} &
\includegraphics[valign=c,width=.22\textwidth]{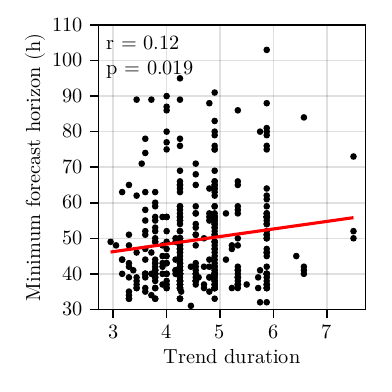} &
\includegraphics[valign=c,width=.22\textwidth]{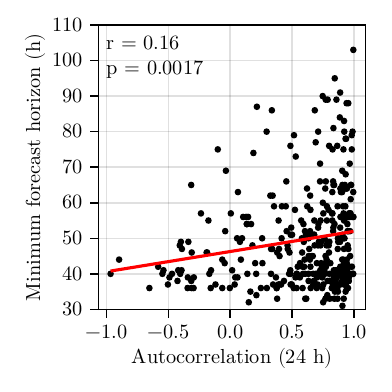} &
\includegraphics[valign=c,width=.22\textwidth]{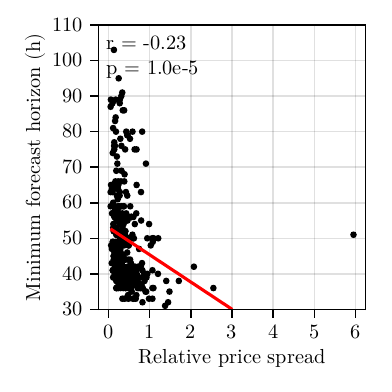}
\\[2mm]

\textbf{60 h} &
\includegraphics[valign=c,width=.22\textwidth]{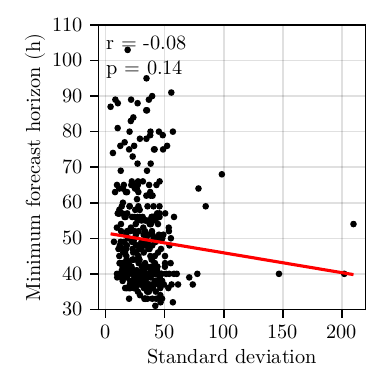} &
\includegraphics[valign=c,width=.22\textwidth]{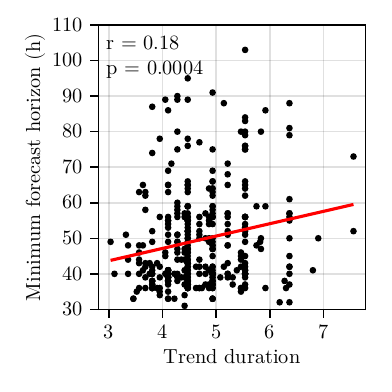} &
\includegraphics[valign=c,width=.22\textwidth]{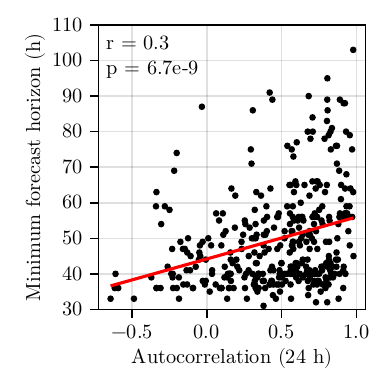} &
\includegraphics[valign=c,width=.22\textwidth]{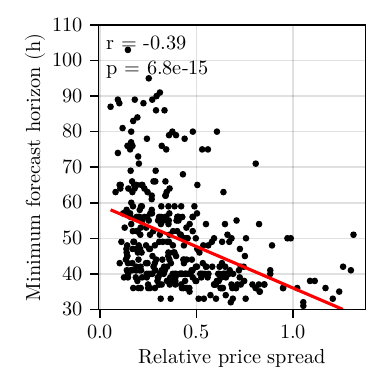}
\\[2mm]

\textbf{80 h} &
\includegraphics[valign=c,width=.22\textwidth]{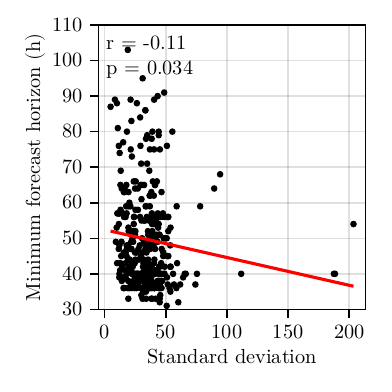} &
\includegraphics[valign=c,width=.22\textwidth]{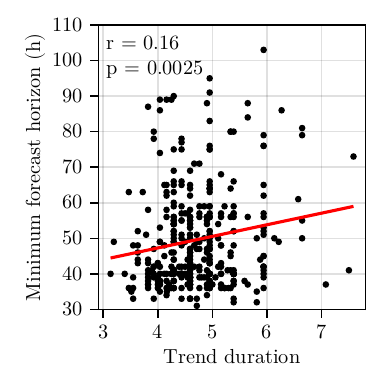} &
\includegraphics[valign=c,width=.22\textwidth]{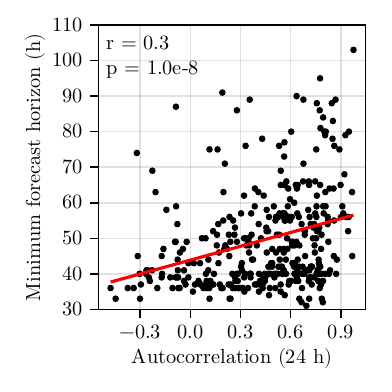} &
\includegraphics[valign=c,width=.22\textwidth]{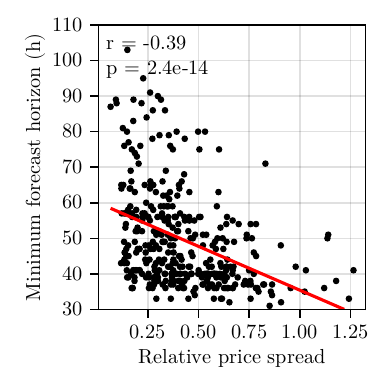}

\end{tabular}
}
\caption{Scatter plots showing the relationship between the minimum forecast horizon and four price metrics computed over analysis windows of 40, 60 and 80 hours for the fast storage.}
\label{fig:scatter}
\end{figure}

We see that the relationship between the price signal and the minimum forecast horizon cannot be captured adequately by any single price metric. Among the metrics considered, the strongest correlation is observed for the relative price spread, with larger spreads generally corresponding to shorter minimum forecast horizons. This is expected, as larger arbitrage opportunities make it easier to identify the optimal charging and discharging strategy within a shorter planning horizon. However, even in this case, the observed correlation remains relatively weak. Furthermore, the strength of the correlation depends on the time window over which the metric is evaluated. If this window is too short, it may exclude future prices that ultimately determine the minimum forecast horizon. Conversely, if it is too long, it includes prices that have no influence because they occur well beyond the relevant planning horizon. This illustrates that the question of the length of the horizon to consider also arises when constructing explanatory metrics from the price signal.

Another observation is that the price values alone are insufficient to explain the minimum forecast horizon. The temporal sequence of charging and discharging opportunities also plays a fundamental role. To illustrate this, we consider the scheduling instance corresponding to 5 January 2024 for the fast storage system, whose original minimum forecast horizon is 88 hours. Starting from the original price signal (``Original''), we generate several alternative price sequences by sorting the prices in increasing order (``Increasing''), sorting them in decreasing order (``Decreasing''), and randomly reshuffling them (``Random''). The prices are modified only after the decision horizon and up to the original minimum forecast horizon, ensuring that all variants contain exactly the same set of prices. Consequently, only their temporal ordering changes. The resulting price signals are shown in Fig.~\ref{fig:price_scn}.
\begin{figure}[ht]
    \centering
    \includegraphics[width=0.55\linewidth]{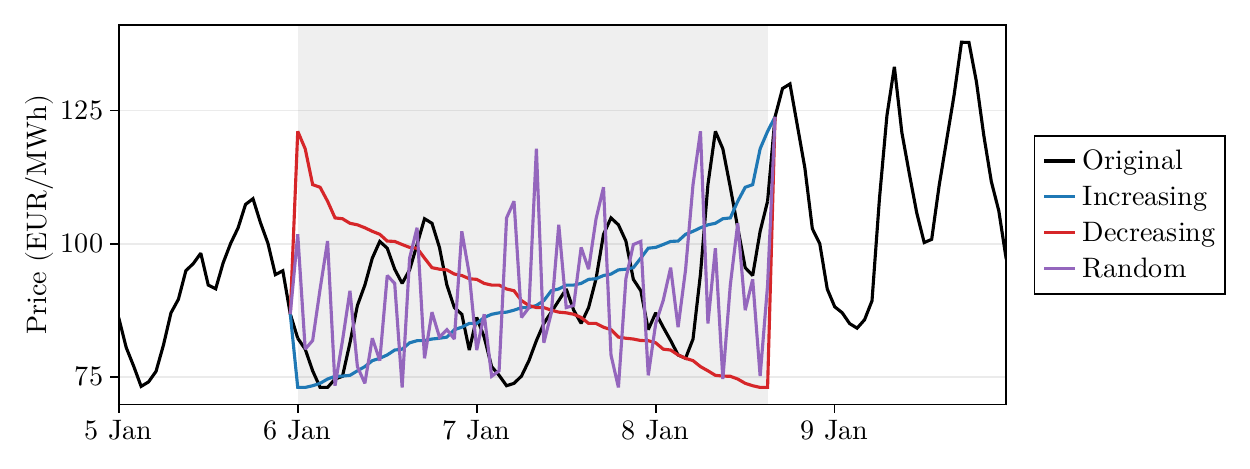}
    \caption{Generated price scenarios. The prices are identical to the original prices outside of the grayed area.}
    \label{fig:price_scn}
\end{figure}

The corresponding minimum forecast horizons are reported in Table~\ref{tab:min_fcst_prices}. Although all scenarios contain identical price values, the minimum forecast horizon ranges from 36 hours to 91 hours depending solely on their ordering. When prices are sorted in increasing order, the optimal strategy quickly becomes apparent: discharge as much as possible during the decision horizon and recharge later at lower prices. The resulting minimum forecast horizon is therefore close to the decision horizon plus the time required to fully discharge the storage system. Sorting the prices in decreasing order increases the minimum forecast horizon because a longer planning horizon is needed to resolve the trade-off between exploiting the high prices available within the decision horizon and preserving stored energy for later periods, where prices are also comparatively high. Nevertheless, the resulting horizon remains shorter than for the original sequence. Other orderings, including the randomly reshuffled sequence, can increase the minimum forecast horizon even further.
\begin{table}[ht]
\centering
\caption{Minimum forecast horizon for the different price scenarios}
\label{tab:min_fcst_prices}
\begin{tabular}{lc}
\hline
Scenario   & Minimum forecast horizon (h) \\ \hline
Original   & 88                           \\
Increasing & 36                           \\
Decreasing & 74                           \\
Random     & 91                           \\ \hline
\end{tabular}
\end{table}

These results demonstrate that the relevant information is not contained in the price values alone, but in the sequence of alternating charging and discharging opportunities created by their temporal evolution. The absence of a simple relationship further justifies the need for a systematic evaluation such as that provided by Algorithm~\ref{alg}. At the same time, it suggests that learning-based approaches could provide a promising avenue for predicting minimum forecast horizons directly from price sequences.

\subsection{Comparison with myopic approaches}

We now discuss the results from the optimization of the schedules of the storage systems, comparing the use of a rolling horizon with a forecast horizon (also referred to as ``Fcst. hor.'') with two other common approaches \cite{Weitzel2018Energy}.
One of these approaches, which we call ``Fixed level,'' is to solve for the next 24 hours only and set the final level to be equal to the initial one. The other is similar but uses an arbitrary rolling horizon of two days, i.e., a planning horizon of 48 hours with a decision horizon of 24 hours (``Hor. 2 days'').

We show the evolution of the state of energy over time in Fig.~\ref{fig:res_soe}. 
For the storage system with a fast charge and high efficiencies, there is not too much difference between the schedules obtained by the rolling-horizon method with a 2-day planning horizon and with a forecast horizon. The state of energy with a fixed level at the end of the day presents fewer variations and rarely discharges completely, as it always has to return to the same level.
For the storage system with a fast charge and low efficiencies, this behavior is more pronounced. 
For the slow storage systems, it is clear that, with a forecast horizon, the minimum or maximum state of energy are reached more often. In the other two cases, we can see that the constraint to end up in the initial level has a strong impact on the resulting schedule.

\begin{figure}[htbp]
    \centering

    \begin{subfigure}{0.32\linewidth}
        \centering
        \includegraphics[width=\linewidth]{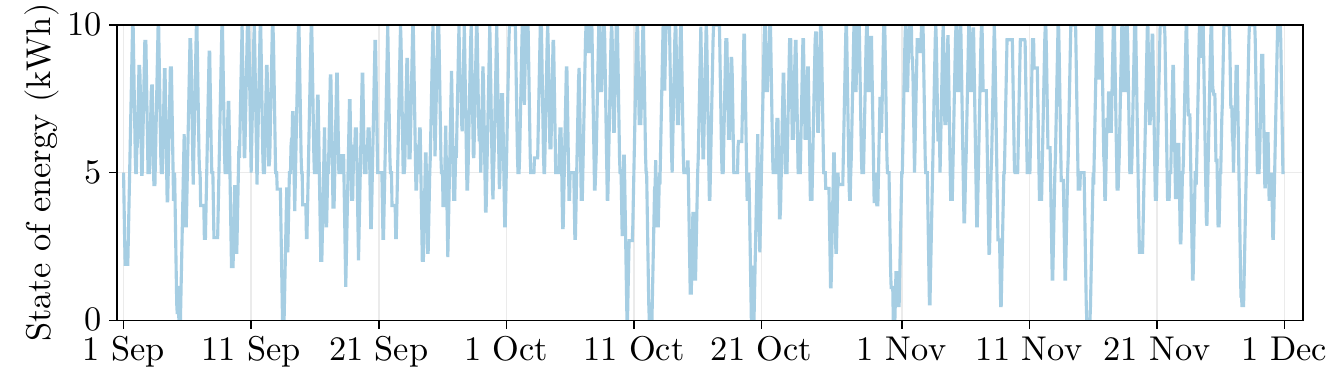}
        \caption{Fast storage: Fixed level}
        \label{fig:res_soe_fast_fix}
    \end{subfigure}
    \hfill
    \begin{subfigure}{0.32\linewidth}
        \centering
        \includegraphics[width=\linewidth]{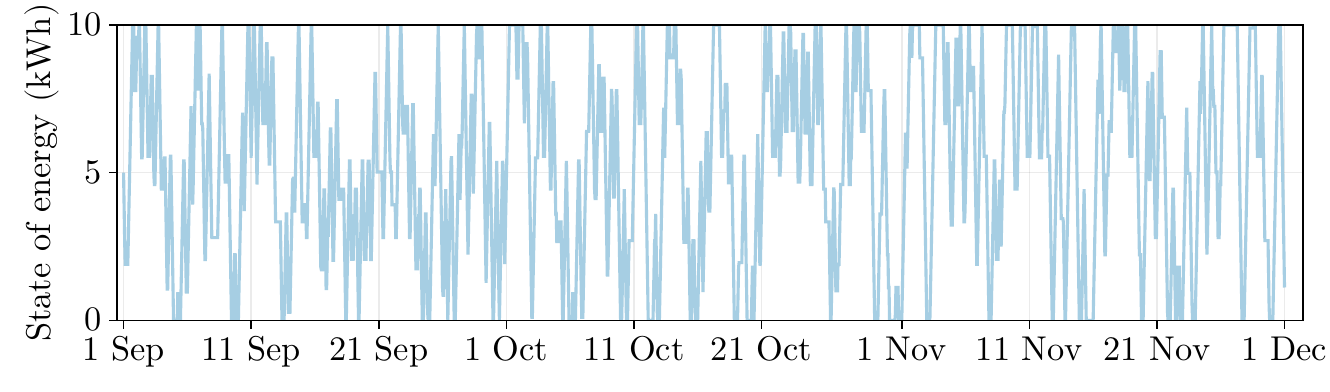}
        \caption{Fast storage: Hor. 2 days}
        \label{fig:res_soe_fast_RH}
    \end{subfigure}
    \hfill
    \begin{subfigure}{0.32\linewidth}
        \centering
        \includegraphics[width=\linewidth]{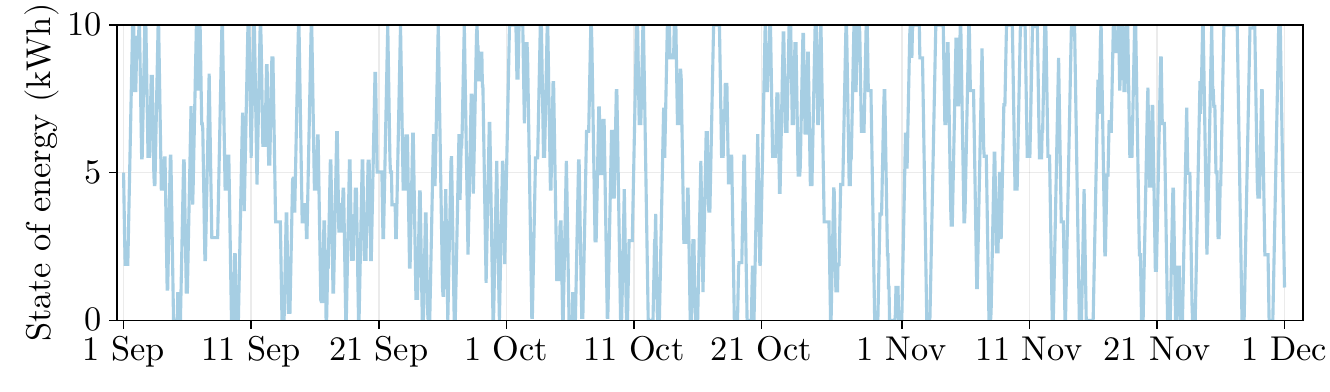}
        \caption{Fast storage: Fcst hor.}
        \label{fig:res_soe_fast_fcst}
    \end{subfigure}

    \vspace{0.4em}

    \begin{subfigure}{0.32\linewidth}
        \centering
        \includegraphics[width=\linewidth]{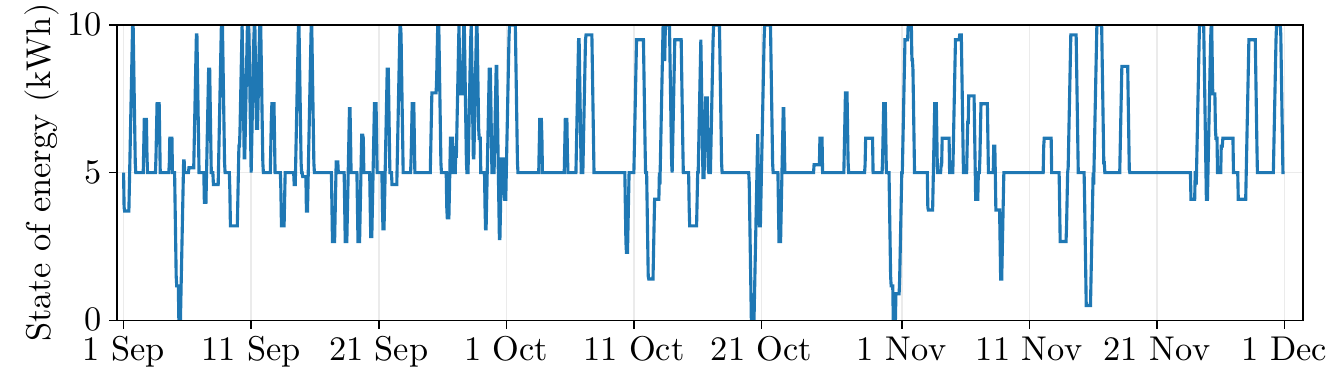}
        \caption{Fast storage (low eff.): Fixed level}
        \label{fig:res_soe_eff_fix}
    \end{subfigure}
    \hfill
    \begin{subfigure}{0.32\linewidth}
        \centering
        \includegraphics[width=\linewidth]{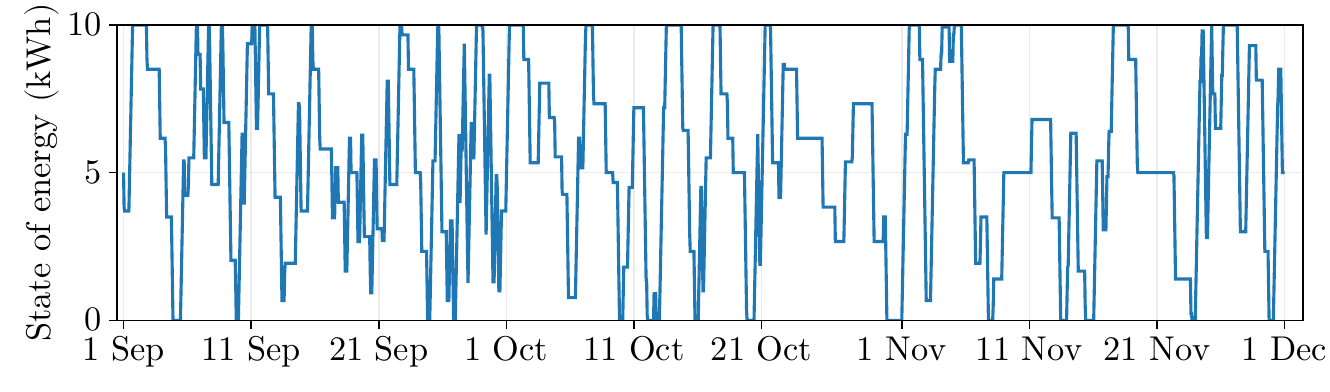}
        \caption{Fast storage (low eff.): Hor. 2 days}
        \label{fig:res_soe_ieff_RH}
    \end{subfigure}
    \hfill
    \begin{subfigure}{0.32\linewidth}
        \centering
        \includegraphics[width=\linewidth]{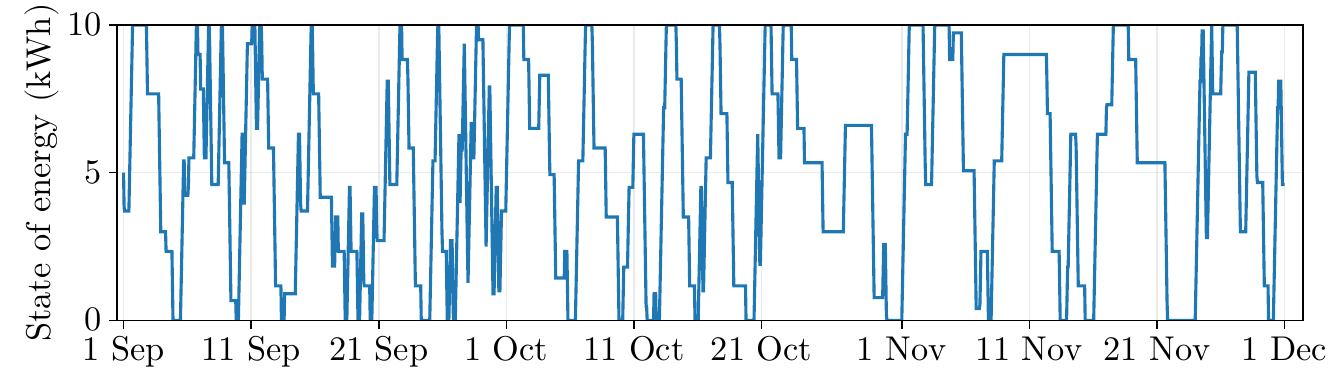}
        \caption{Fast storage (low eff.): Fcst. hor.}
        \label{fig:res_soe_eff_fcst}
    \end{subfigure}

    \vspace{0.4em}

    \begin{subfigure}{0.32\linewidth}
        \centering
        \includegraphics[width=\linewidth]{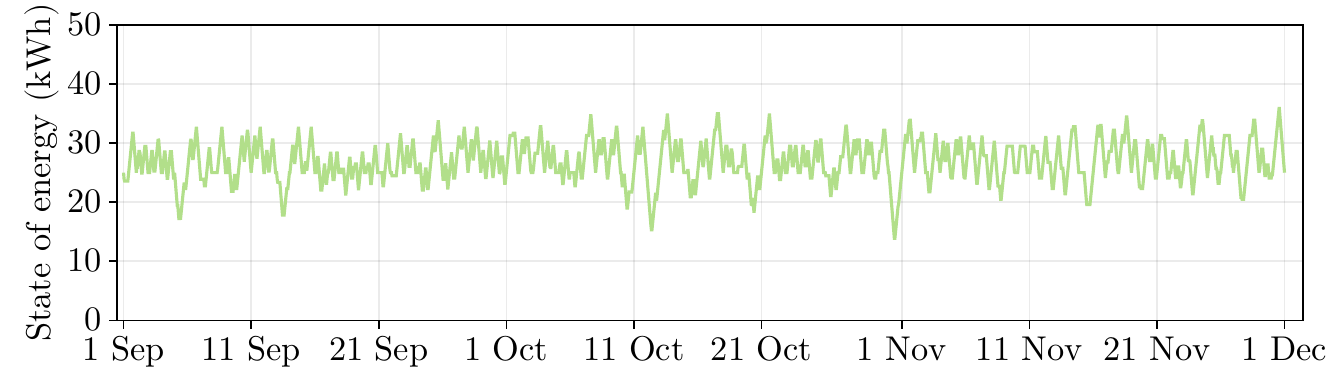}
        \caption{Slow storage: Fixed level}
        \label{fig:res_soe_slow_fix}
    \end{subfigure}
    \hfill
    \begin{subfigure}{0.32\linewidth}
        \centering
        \includegraphics[width=\linewidth]{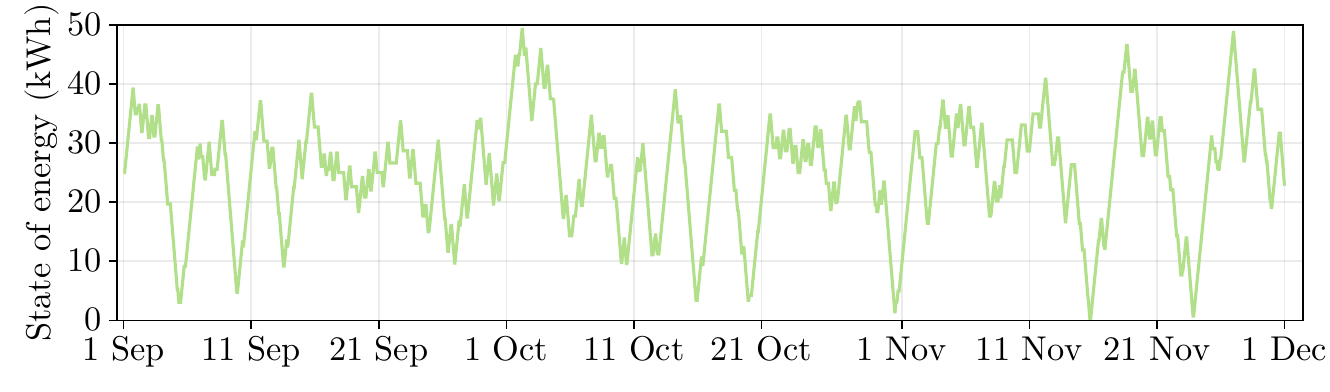}
        \caption{Slow storage: Hor. 2 days}
        \label{fig:res_soe_slow_RH}
    \end{subfigure}
    \hfill
    \begin{subfigure}{0.32\linewidth}
        \centering
        \includegraphics[width=\linewidth]{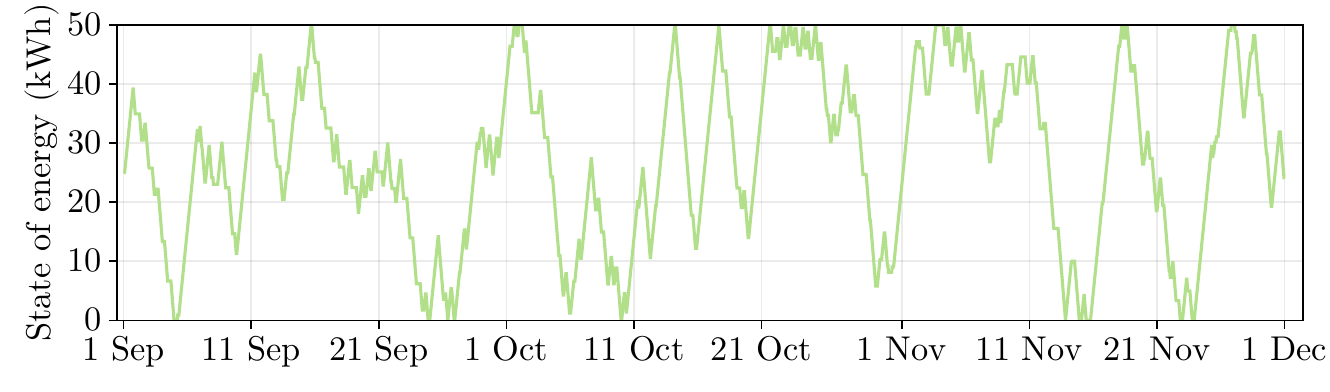}
        \caption{Slow storage: Fcst. hor.}
        \label{fig:res_soe_slow_fcst}
    \end{subfigure}

    \vspace{0.4em}

    \begin{subfigure}{0.32\linewidth}
        \centering
        \includegraphics[width=\linewidth]{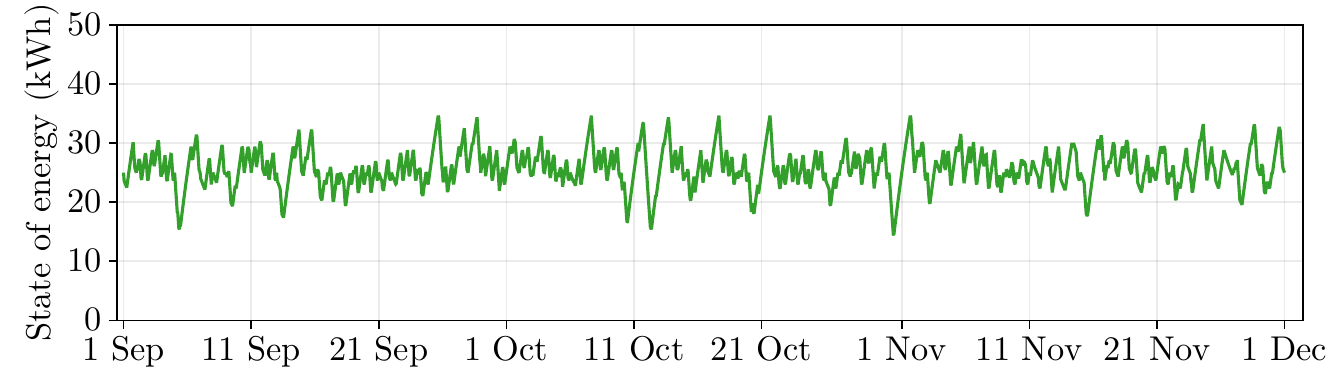}
        \caption{Slow storage (leakage): Fixed level}
        \label{fig:res_soe_leak_fix}
    \end{subfigure}
    \hfill
    \begin{subfigure}{0.32\linewidth}
        \centering
        \includegraphics[width=\linewidth]{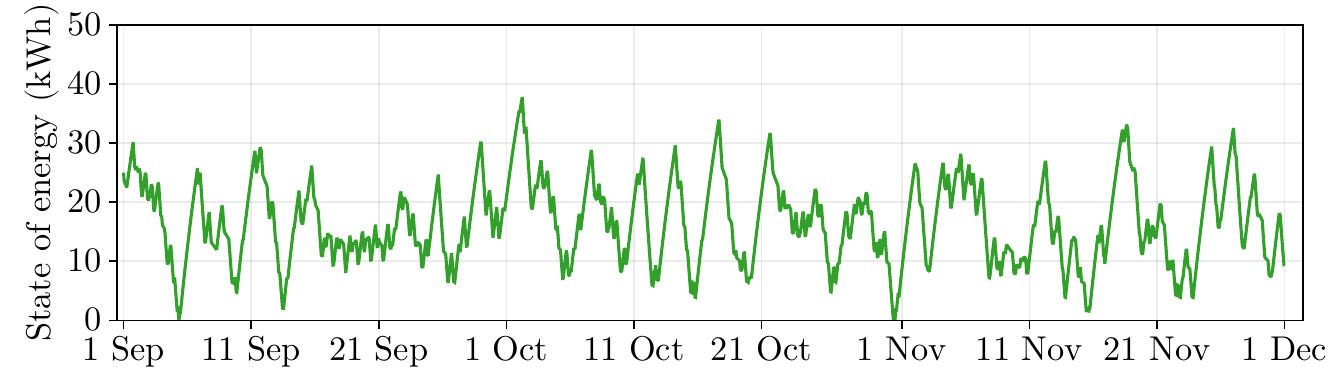}
        \caption{Slow storage (leakage): Hor. 2 days}
        \label{fig:res_soe_leak_RH}
    \end{subfigure}
    \hfill
    \begin{subfigure}{0.32\linewidth}
        \centering
        \includegraphics[width=\linewidth]{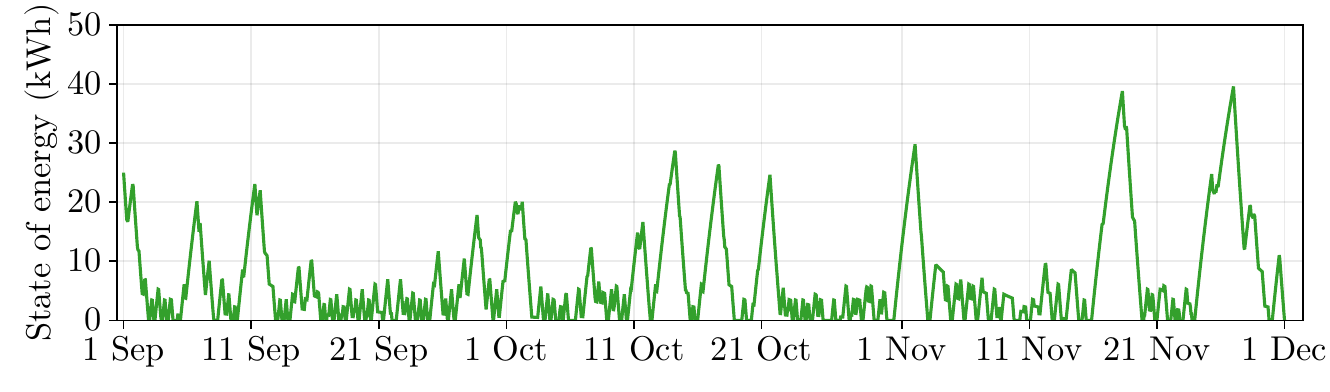}
        \caption{Slow storage (leakage): Fcst. hor.}
        \label{fig:res_soe_leak_fcst}
    \end{subfigure}

    \caption{State of energy for the different test cases}
    \label{fig:res_soe}
\end{figure}

Results in terms of the total sum of the values of the objective function are given in Table~\ref{tab:res_all}.
This table also shows the results for a planning horizon of 5 days (``Hor. 5 days''), which corresponds to the time necessary for the slow storage to complete a full charge and discharge.
The numbers in parentheses represent the profit loss compared to the case with a forecast horizon. The advantage of using a forecast horizon is clear. It is less critical for faster storage systems, and in this case, it is the most critical for the slow storage system with leakage. This is due to the fact that the optimal strategy is to operate around the minimum level, while the other approaches push the operation closer to half of the capacity instead.
Compared to the profit obtained when using a forecast horizon, the loss in profit can be as much as $127\%$ with a fixed level and $49\%$ with a horizon of two days.
These results also illustrate that using a planning horizon that is too short can even lead to a negative profit.
In these examples, the solution obtained with a planning horizon of 5 days is much better; however, for the fast storage with low efficiencies and for the slow storage systems, it is still suboptimal.
Finally, it appears essential to properly account for charging and discharging efficiencies, as the difference between the results of the first two examples suggests.

\begin{table}[ht]
    \centering
    \caption{Profit (€) for the different test cases}
    \begin{tabular}{lrrrr}
    \hline
    & \multicolumn{1}{c}{Fixed level} & \multicolumn{1}{c}{Hor. 2 days} & \multicolumn{1}{c}{Hor. 5 days} & \multicolumn{1}{c}{\begin{tabular}[c]{@{}c@{}}Fcst. hor. \\\end{tabular}} \\ \hline
    Fast storage              & 35.11 (16\%)   & 40.82 (0.1\%) & 40.85 (0\%) & \textbf{40.85} \\
    Fast storage (low eff.) & 10.58 (41\%)   & 16.99 (6\%) & 18.01 (0.1\%)  & \textbf{18.04} \\
    Slow storage & 37.80 (36\%)   & 53.36 (10\%) & 59.08 (0.2\%) & \textbf{59.23}  \\
    Slow storage (leakage)    & -10.86 (127\%) & 20.79 (49\%) & 40.69 (0.2\%) & \textbf{40.76} \\ \hline
    \end{tabular}
    \label{tab:res_all}
\end{table}

\subsection{Evaluation of suboptimality} \label{sec:res_subopt}

We use \eqref{eq:subopt1} to compute an upper bound on the suboptimality gap for the slow storage with leakage.
As $\underline{C}$ and $\overline{C}$, we use the minimum and maximum prices imposed on the Nordpool day-ahead market, which are -500~€/MWh and 4000~€/MWh, respectively. 
We evaluate suboptimality for September 1, 2024, and analyze how it evolves as we increase the planning horizon.
The storage level at the end of the decision horizon $s_H$ is chosen such that $\overline{s}_H \leq s_H \leq \underline{s}_H$, and such that the profit over the decision horizon is maximized, i.e., $Z\up{DH} = Z\up{opt,DH}$.
The plots in Fig.~\ref{fig:gap} and Fig.~\ref{fig:bound} respectively illustrate the evolution of the gap between $\overline{s}_H$ and $\underline{s}_H$, and the evolution of the upper bound on suboptimality, for an increasing planning horizon. Both follow the same trend.

\begin{figure}[htbp]
    \centering

    \begin{subfigure}{0.48\linewidth}
        \centering
        \includegraphics[width=\linewidth]{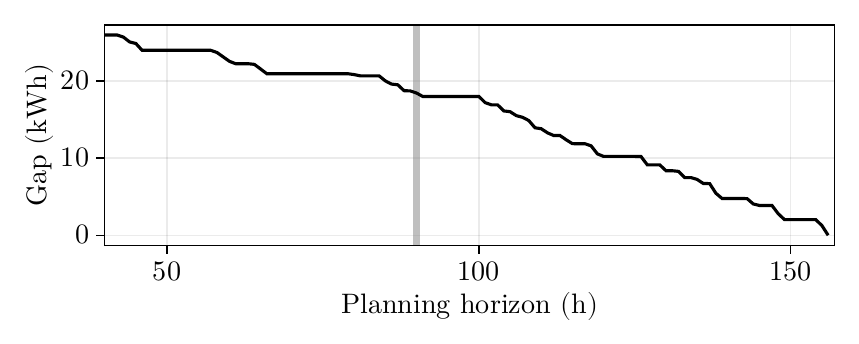}
        \caption{Gap between $\overline{s}_H$ and $\underline{s}_H$}
        \label{fig:gap}
    \end{subfigure}
    \hfill
    \begin{subfigure}{0.48\linewidth}
        \centering
        \includegraphics[width=\linewidth]{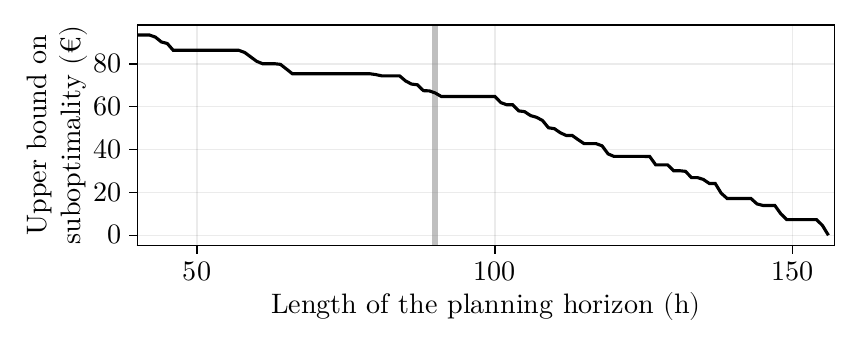}
        \caption{Upper bound on the suboptimality gap}
        \label{fig:bound}
    \end{subfigure}

    \caption{Suboptimality depending on the length of the planning horizon. The gray line highlights the results for a planning horizon of 90 hours.}
    \label{fig:res_subopt}
\end{figure}

To provide more insight, we further analyze the results obtained for a planning horizon of 90 hours, which are highlighted in Fig.~\ref{fig:res_subopt}. In this case, the upper bound on suboptimality is 64.74~\texteuro, or 2400\%.
The actual suboptimality is obtained by continuing to solve the scheduling problem over the following days. For each of these next days, a forecast horizon is used. We stop when the resulting state of energy at the end of a day for the problem with a suboptimal decision on the first day is equal to that of the infinite-horizon problem, which occurs after 3 more days. Comparing the sum of the objective values over each decision horizon, we obtain a suboptimality of 0, meaning that the final level obtained at the end of the decision horizon is actually optimal.
If instead of the Nordpool minimum and maximum prices, we use the minimum and maximum prices observed between 2019 and 2023 for that zone, which are -440.1 \texteuro/MWh and 871 \texteuro/MWh, the upper bound on the suboptimality gap becomes 14.45~\texteuro, or 536\%.
In both cases, we note that the bound is quite conservative. Indeed, not only does it assume that extreme prices are going to arise, but also that they will arise at a moment when the storage system is unable to react to them. This would be the case, for example, if the storage system were empty at the end of the day and the maximum possible price occurs in the first hour of the next day. If it were not the first hour, then the storage system could first charge and then discharge at the maximum price later. The upper bound could be refined to include these considerations.
\section{Practical use of the proposed framework}\label{sec:prac}

The results presented in this paper provide a systematic framework for evaluating planning horizons in rolling-horizon (or receding-horizon) approaches to energy storage scheduling. The proposed methodology complements existing operational scheduling practices by determining whether a chosen planning horizon is sufficiently long and by providing guidance for selecting planning horizons in future implementations. This section summarizes how the theoretical results translate into practical tools and clarifies the class of problems to which they apply.

\subsection{Scope of applicability}

The theoretical results presented in this paper apply to energy storage scheduling problems satisfying the assumptions of the model introduced in Sect.~\ref{sec:model}. In particular, the storage system is represented through its state of energy, charging and discharging limits, charging and discharging efficiencies, and leakage.

More generally, the framework applies to optimization problems satisfying the conditions described in Sect.~\ref{sec:complex}.
This includes formulations involving multiple storage systems, other sources of production and loads, electricity networks represented using a linear power flow model (DC-OPF). It also covers behind-the-meter systems combining storage with local generation and demand. The pricing scheme is not restricted to wholesale electricity markets and may also represent time-of-use tariffs, flat tariffs, or demand charges. The key requirement is that, in the objective function and in each constraint expression involving storage operation, the charging and discharging variables appear with coefficients of the same magnitude and opposite signs, as formally characterized in Sect.~\ref{sec:complex}. This condition also allows discounted objective functions. While the present work does not explicitly study the effect of discounting on forecast horizons, discounting reduces the influence of distant periods on the objective function and would therefore be expected to promote the existence of forecast horizons and reduce their length. A theoretical characterization of the influence of discounting on forecast horizon existence and length is left for future work.

The framework does not cover formulations in which the storage system is price-maker. Extending the theory to such settings is an interesting direction for future work. For optimization problems outside the proposed framework, the theoretical guarantees no longer apply. Nevertheless, the methodology developed here may still provide useful qualitative insights into the influence of planning horizon selection.

\subsection{Operational use}

For a given application, the decision horizon is typically imposed by operational requirements (for example, 24 hours). The remaining design choice is the length of the planning horizon over which the optimization problem is solved. We refer to Sect.~\ref{sec:def} for the formal definition of the different types of horizons.

In Fig. \ref{fig:framework}, we summarize how to use the results of the paper. Details are given below for the online part, and in the next subsection for the offline part.

\begin{figure}[ht]
    \centering
\includegraphics[width=0.99\linewidth]{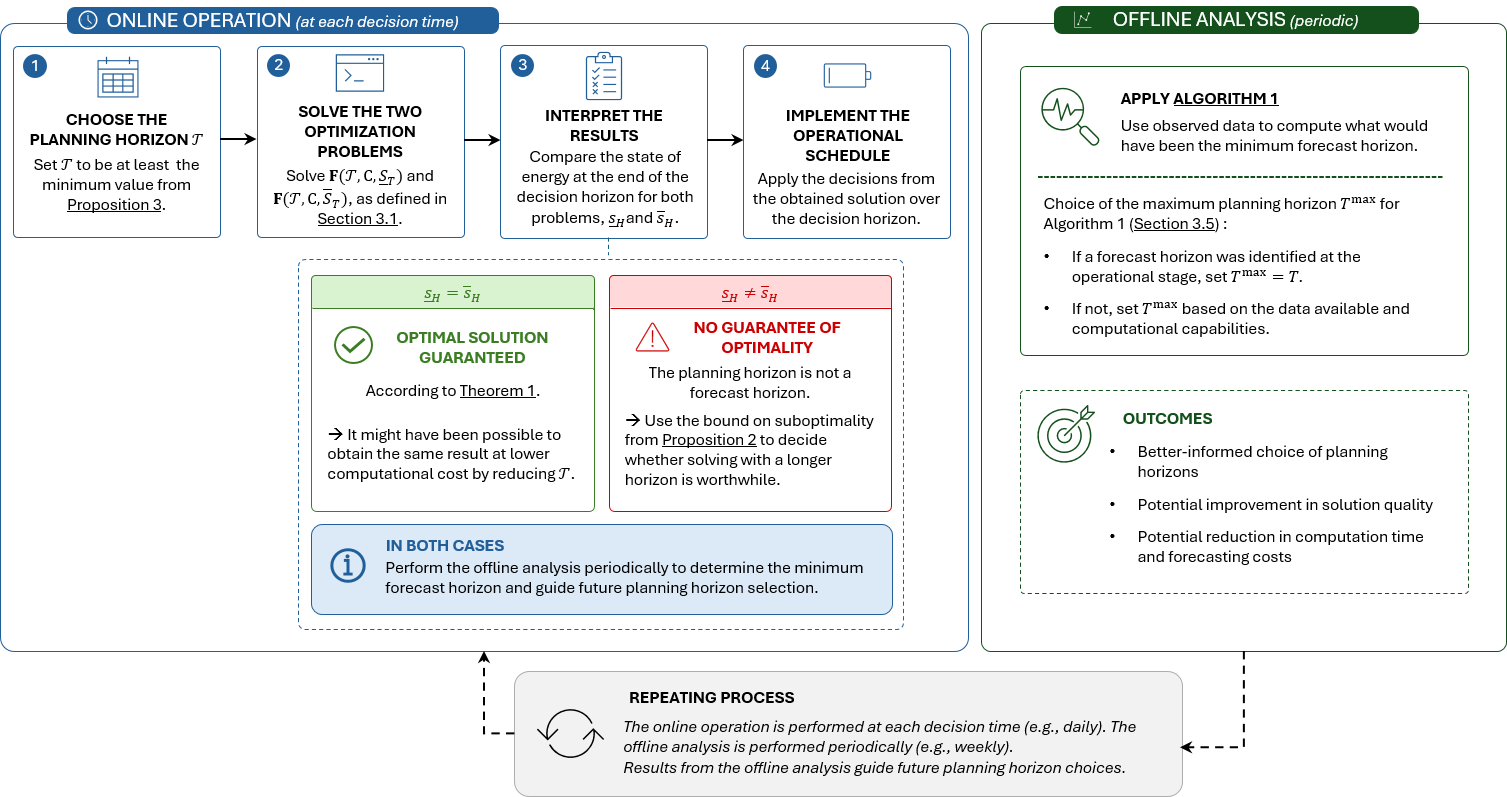}
    \caption{Summary of how to use the results of the paper for scheduling energy storage systems}
    \label{fig:framework}
\end{figure}

Given a planning horizon, two versions of the scheduling problem are solved, as described in Sect.~\ref{sec:criteria}. In the first one, the terminal state of energy is fixed at the minimum reachable level, and in the other one, it is fixed at the maximum reachable level. 
The resulting states of energy at the end of the decision horizon are then compared.
If the resulting states of energy differ, the chosen planning horizon is not a forecast horizon. In this case, a longer planning horizon could improve the solution. Proposition~\ref{prop2} can then be used to quantify an upper bound on the resulting suboptimality. The optimal solution is known to lie between the two extreme solutions considered. Developing methods to determine appropriate terminal conditions when a forecast horizon cannot be reached constitutes an interesting direction for future research.
If these two states of energy coincide, Theorem~\ref{theorem1} guarantees that the chosen planning horizon is a forecast horizon. Extending the planning horizon further cannot improve the operational decisions over the decision horizon. However, it is possible that a shorter planning horizon would have produced the same operational decisions while reducing computational effort and, potentially, forecasting costs.
For these reasons, we encourage practitioners to conduct a study of the minimum forecast horizon a posteriori, which we describe next.

For optimization problems involving multiple storage systems, the same procedure applies by fixing the terminal state of energy of every storage system simultaneously to either its minimum or its maximum reachable level. As discussed in Sect.~\ref{sec:complex}, the planning horizon is a forecast horizon for the complete problem if it is a forecast horizon for each storage subsystem.

\subsection{Offline evaluation of planning horizons}

Beyond evaluating individual scheduling instances, the proposed framework can be used offline to characterize the planning horizon requirements of a given application.
Applying Algorithm~\ref{alg} to representative historical data determines the minimum forecast horizon for each scheduling instance. This information can then be used to analyze how forecasting requirements vary over time and across storage technologies.
The resulting distribution of minimum forecast horizons can guide the selection of planning horizons in future implementations. Depending on the desired balance between computational effort, forecasting requirements, and robustness, practitioners may choose planning horizons corresponding to, for example, the maximum observed minimum forecast horizon, a high percentile of the distribution, or another representative value.
The numerical study presented in Sect.~\ref{sec:ex} illustrates this type of analysis by evaluating the minimum forecast horizon for storage systems with different technical characteristics. The results demonstrate how the proposed framework can be used to assess the influence of storage properties on planning horizon requirements and to compare the performance of arbitrary planning horizons with that of forecast horizons.
\section{Conclusion}\label{sec:ccl}

In this paper, we consider the problem of determining the minimum forecast horizon for storage scheduling problems to be used in a rolling-horizon approach. We demonstrate that the existence of a forecast horizon is not guaranteed.
We formally introduce an easily verifiable condition to check if a chosen planning horizon is long enough. It involves solving two problems, where the final state of energy is set to its minimum and maximum possible value, respectively. If the state of energy at the end of the decision horizon is the same for these two problems, then the planning horizon is a forecast horizon. 
We propose an iterative algorithm to find the minimum forecast horizon. This procedure is initialized using a lower bound on the minimum forecast horizon that is theoretically computed using the features of the storage device and is valid for any price profile.

Numerical experiments show that several factors related to the storage system characteristics influence the length of the minimum forecast horizon. These include its charging and discharging efficiencies, the duration of charge and discharge, and the rate of self-discharge or leakage. Another key factor is how prices vary, which is not straightforward to characterize.
Comparing the use of a forecast horizon to common arbitrary approaches, we demonstrate the advantage of forecast horizons. For some storage characteristics, using a planning horizon much shorter than the minimum forecast horizon can even lead to a negative profit.

Several research directions naturally follow from this work. The first one is to extend the framework to stochastic settings, and study how forecast uncertainty affects the length of forecast horizons. A second direction is the development of operational methods that dynamically select planning horizons, potentially using learning-based approaches. A complementary extension is to determine the terminal conditions (final state of energy) that are expected to minimize the resulting suboptimality when sufficiently long planning horizons are impractical because of computational or forecasting limitations. Finally, integrating forecast quality, forecasting costs, and computational effort into the planning horizon selection process would further strengthen the practical applicability of the proposed framework.

\begin{appendices}
\numberwithin{equation}{section}
\section{Proof of Theorem~\ref{theorem1}} \label{app:proof1}

In all our proofs, we use the following expression of the state of energy at any time period $t_2$ with the state of energy at any previous time period $t_1$, obtained by merging constraints \eqref{eq:update} between $t_1$ and $t_2$:
\begin{equation} \label{eq:s_calc}
    s_{t_2} = \rho^{t_2 - t_1} s_{t_1} + \Delta t\sum_{t=t_1+1}^{t_2} \rho^{t_2 - t} \left( \eta\up{C} p_t\up{C} - \frac{1}{\eta\up{D}} p_t\up{D} \right), \quad \forall (t_1,t_2) \in \mathcal{T}, t_1 \leq t_2, s_0 = S\up{init}.
\end{equation}

A first useful result is that the state of energy obtained when setting the final level to the maximum reachable level is an upper bound for an optimal solution to the problem with any final level. Similarly, the state of energy obtained when setting the final level to the minimum reachable level is a lower bound for an optimal solution to the problem with any final level.

\begin{lemma} \label{lemma1}
    The following is true
    \begin{enumerate}
        \item Given $\mathbf{\overline{x}} \in \overline{\mathcal{X}}$ and $S\up{end}$ , with $\underline{S}_T \leq S\up{end} \leq \overline{S}_T$, there exists $\mathbf{x}^* \in \mathcal{X}^*$ such that ${s}_t^* \leq \overline{s}_t$, $\forall t \in \mathcal{T}$.
        \item Given $\mathbf{\underline{x}} \in \underline{\mathcal{X}}$ and $S\up{end}$ , with $\underline{S}_T \leq S\up{end} \leq \overline{S}_T$, there exists $\mathbf{x}^* \in \mathcal{X}^*$ such that ${s}_t^* \geq \underline{s}_t$, $\forall t \in \mathcal{T}$.
    \end{enumerate}
\end{lemma}

\begin{proof}
    We consider any solution $\mathbf{\overline{x}} \in \overline{\mathcal{X}}$, and any final level $S\up{end}$, with $\underline{S}_T \leq S\up{end} \leq \overline{S}_T$.
    We suppose that for all solutions of $\textbf{F}(\mathcal{T}, \textbf{C}, S\up{end})$, there exists $t \in \mathcal{T}$ such that $s_{t}^* > \overline{s}_t$ and show a contradiction.
    
    Consider any of these solutions, $\mathbf{x}^* \in \mathcal{X}^*$.
    We identify $t_1 \in \mathcal{T}$, such that $s_{t}^* \leq \overline{s}_{t}$, $\forall t < t_1$ and $s_{t_1}^* > \overline{s}_{t_1}$.
    Since $s_T^* = S\up{end}$, $\overline{s}_T = \overline{S}_T$ and $S\up{end} \leq \overline{S}_T$, $\exists \, t_2 \in \mathcal{T}, \, t_2> t_1$ such that $s_{t}^* > \overline{s}_{t}$, $\forall t \in [t_1, t_2-1]$ and $s_{t_2}^* \leq \overline{s}_{t_2}$.

    We build solutions $\mathbf{x}'$ and $\overline{\mathbf{x}}'$, based on $\mathbf{x}^*$ and $\overline{\mathbf{x}}$ respectively, and derive some properties of the problem.

    We first build $\mathbf{x}'$ based on $\mathbf{x}^*$. We modify the solution at $t_1$, decreasing charge if $p\up{C*}_{t_1}>0$ or increasing discharge if $p\up{C*}_{t_1}=0$, which implies $p\up{D*}_{t_1}\geq0$. Therefore, the complementarity constraint at $t_1$ still holds.

    \textit{\textbf{If $p\up{C*}_{t_1}>0$:}} \\ 
    We decrease the charge at $t_1$ by a small quantity ${\epsilon} > 0$, with ${\epsilon} \leq {p}\up{C*}_{t_1}$, such that ${p}\up{C'}_{t_1}={p}\up{C*}_{t_1}-{\epsilon}$ and the bounds on charge are respected.
    We apply \eqref{eq:s_calc} between $t_1$ and $t \in [t_1,t_2-1]$ to obtain the modified state of energy at $t$:
    \begin{equation}
        {s}_t' = \rho^{t - t_1} s_{t_1}' + \Delta t \sum_{k=t_1+1}^{t} \rho^{t - k} \left( \eta\up{C} p_k\up{C'} - \frac{1}{\eta\up{D}} p_k\up{D'} \right) = {s}_t^* - \rho^{t - t_1} \Delta t \, \eta\up{C} {\epsilon}.
    \end{equation}
    We choose ${\epsilon}$ such that the lower bound on the state of energy is respected, i.e. ${\epsilon} \leq \frac{{s}_t^* - \underline{S}}{\rho^{t - t_1} \Delta t \, \eta\up{C}}$, $\forall t \in [t_1,t_2-1]$, which is possible since $\underline{S} \leq \overline{s}_t < s_t^*$, $ \forall t \in [t_1,t_2-1]$. 
    Since ${s}_t^* \leq \overline{S}$, the upper bound is still respected.

    We now modify the solution at $t_2$ to return on the same trajectory, decreasing discharge if $p\up{D*}_{t_2}>0$ or increasing charge if $p\up{D*}_{t_2}=0$ and $p\up{C*}_{t_2}\geq0$. Therefore, the complementarity constraint at $t_2$ still holds.

    \begin{itemize}
        \item \textit{If $p\up{D*}_{t_2}>0$}: We decrease discharge at $t_2$ by $\rho^{t_2 - t_1} \eta\up{C} \eta\up{D} \epsilon$, with ${\epsilon} \leq \frac{{p}\up{D*}_{t_2}}{\rho^{t_2 - t_1} \eta\up{C} \eta\up{D}} $, such that ${p}\up{D'}_{t_2}={p}\up{D*}_{t_2}-\rho^{t_2 - t_1} \eta\up{C} \eta\up{D} \epsilon$ and the bounds on discharge are respected. At $t_2$, the modified state of energy is
        \begin{equation}
            {s}_{t_2}' = \rho {s}_{t_2-1}^* - \rho^{t_2 - t_1} \Delta t \, \eta\up{C} {\epsilon} +\Delta t \, \eta\up{C} {p}\up{C*}_{t_2} - \Delta t \frac{1}{\eta\up{D}} {p}\up{D*}_{t_2} + \Delta t \frac{1}{\eta\up{D}} \rho^{t_2 - t_1} \eta\up{C} \eta\up{D} \epsilon = {s}_{t_2}^*.   
        \end{equation}
        We are back on the same trajectory, which completes the demonstration that the new solution is feasible.
        The change in the objective function compared to $\mathbf{x}^*$ is $\Delta t \, C_{t_1} \epsilon - \Delta t \, \rho^{t_2 - t_1} \eta\up{C} \eta\up{D} C_{t_2} \epsilon$. Since $\mathbf{x}'$ cannot be better than $\mathbf{x}^*$, we derive
        \begin{equation} \label{eq:star_pc_pd}
                C_{t_1} \leq \rho^{t_2 - t_1} \eta\up{C} \eta\up{D} C_{t_2}.
        \end{equation}
        \item \textit{If $p\up{D*}_{t_2}=0$ and $p\up{C*}_{t_2}\geq0$}: We increase charge at $t_2$ by $\rho^{t_2 - t_1} \epsilon$, with ${\epsilon} \leq \frac{\overline{P}\up{C} - p\up{C*}_{t_2}}{\rho^{t_2 - t_1}}$, such that ${p}\up{C'}_{t_2}={p}\up{C*}_{t_2}+\rho^{t_2 - t_1} \epsilon$ and the bounds on charge are respected.
        Note that ${\epsilon} >0$ is guaranteed by $p\up{C*}_{t_2} < \overline{P}\up{C}$. Indeed, since $\overline{s}_{t_2} < {s}_{t_2}^*$, $\overline{s}_{t_2} \geq {s}_{t_2}^*$ would not be possible otherwise.
        At $t_2$, the modified state of energy is
        \begin{equation}
            {s}_{t_2}' = \rho {s}_{t_2-1}^* - \rho^{t_2 - t_1} \Delta t \, \eta\up{C} {\epsilon} +\Delta t \, \eta\up{C} {p}\up{C*}_{t_2} +  \Delta t \, \eta\up{C} \rho^{t_2 - t_1} {\epsilon} - \Delta t \frac{1}{\eta\up{D}} {p}\up{D*}_{t_2} = {s}_{t_2}^*.
        \end{equation}
        We are back on the same trajectory, which completes the demonstration that the new solution is feasible.
        The change in the objective function compared to $\mathbf{x}^*$ is $\Delta t \, C_{t_1} \epsilon - \Delta t \, \rho^{t_2 - t_1} C_{t_2} \epsilon$. Since $\mathbf{x}'$ cannot be better than $\mathbf{x}^*$, we derive
        \begin{equation} \label{eq:star_pc_pc}
            C_{t_1} \leq \rho^{t_2 - t_1} C_{t_2}.
        \end{equation}
    \end{itemize}

   \textit{\textbf{If $p\up{C*}_{t_1}=0$ and $p\up{D*}_{t_1}\geq0$:}} \\ 
   We increase the discharge at $t_1$ by a small quantity ${\epsilon} > 0$, with ${\epsilon} \leq \overline{P}\up{D} - {p}\up{D*}_{t_1}$, such that ${p}\up{D'}_{t_1}={p}\up{D*}_{t_1}+{\epsilon}$ and the bounds on discharge are respected. Note that ${\epsilon} >0$ is guaranteed by $p\up{D*}_{t_1} < \overline{P}\up{D}$. Indeed, since ${s}_{t_1}^* < \overline{s}_{t_1}$, ${s}_{t_1}^* \geq \overline{s}_{t_1}^*$ would not be possible otherwise.
    We apply \eqref{eq:s_calc} between $t_1$ and $t \in [t_1,t_2-1]$ to obtain the modified state of energy at $t$:
    \begin{equation}
        {s}_t' = \rho^{t - t_1} s_{t_1}' + \Delta t \sum_{k=t_1+1}^{t} \rho^{t - k} \left( \eta\up{C} p_k\up{C'} - \frac{1}{\eta\up{D}} p_k\up{D'} \right) = {s}_t^* - \rho^{t - t_1} \Delta t \frac{1}{\eta\up{D}} {\epsilon}.   
    \end{equation}
    We choose ${\epsilon}$ such that the lower bound on the state of energy is respected, i.e. ${\epsilon} \leq  \frac{\eta\up{D}({s}_t^* - \underline{S})}{\rho^{t - t_1} \Delta t}$, $\forall t \in [t_1,t_2-1]$, which is possible since $\underline{S} \leq \overline{s}_t < s_t^*$, $ \forall t \in [t_1,t_2-1]$. 
    Since ${s}_t^* \leq \overline{S}$, the upper bound is still respected.

    We then modify the solution at $t_2$ in a similar way as before.

    \begin{itemize}
        \item \textit{If $p\up{D*}_{t_2}>0$}: We decrease discharge at $t_2$ by $\rho^{t_2 - t_1} \epsilon$, with ${\epsilon} \leq \frac{p\up{D*}_{t_2}}{\rho^{t_2 - t_1}}$, such that ${p}\up{D'}_{t_2}={p}\up{D*}_{t_2}-\rho^{t_2 - t_1}\epsilon$ and the bounds on discharge are respected.
        At $t_2$, the modified state of energy is
        \begin{equation}
            {s}_{t_2}' = \rho {s}_{t_2-1}^* - \rho^{t_2 - t_1} \Delta t \frac{1}{\eta\up{D}} {\epsilon} +\Delta t \, \eta\up{C} {p}\up{C*}_{t_2} - \Delta t \frac{1}{\eta\up{D}} {p}\up{D*}_{t_2} + \Delta t \frac{1}{\eta\up{D}} \rho^{t_2 - t_1} \epsilon = {s}_{t_2}^*.  
        \end{equation}
        We are back on the same trajectory, which completes the demonstration that the new solution is feasible.
        The change in the objective function compared to $\mathbf{x}^*$ is $ \Delta t \, C_{t_1} \epsilon - \Delta t \, \rho^{t_2 - t_1} C_{t_2} \epsilon$. Since $\mathbf{x}'$ cannot be better than $\mathbf{x}^*$, we derive
        \begin{equation} \label{eq:star_pd_pd}
            C_{t_1} \leq \rho^{t_2 - t_1} C_{t_2}.
        \end{equation}
        \item \textit{If $p\up{D*}_{t_2}=0$ and $p\up{C*}_{t_2}\geq 0$}: We increase charge at $t_2$ by $\frac{ \rho^{t_2 - t_1} \epsilon}{\eta\up{C} \eta\up{D}}$, with ${\epsilon} \leq \frac{\eta\up{C} \eta\up{D} (\overline{P}\up{C} - p\up{C*}_{t_2})}{\rho^{t_2 - t_1}}$, such that ${p}\up{C'}_{t_2}={p}\up{C*}_{t_2}+\frac{\rho^{t_2 - t_1} \epsilon}{\eta\up{C} \eta\up{D}}$. The bounds on discharge are respected.
        At $t_2$, the modified state of energy is
        \begin{equation}
            {s}_{t_2}' = \rho {s}_{t_2-1}^* - \rho^{t_2 - t_1} \Delta t \frac{1}{\eta\up{D}} {\epsilon} +\Delta t \, \eta\up{C} {p}\up{C*}_{t_2} +\Delta t \, \eta\up{C} \frac{\rho^{t_2 - t_1} \epsilon}{\eta\up{C} \eta\up{D}} - \Delta t \frac{1}{\eta\up{D}} {p}\up{D*}_{t_2} ={s}_{t_2}^*.
        \end{equation}
        We are back on the same trajectory, which completes the demonstration that the new solution is feasible.
        The change in the objective function compared to $\mathbf{x}^*$ is $ \Delta t \, C_{t_1} \epsilon - \Delta t \frac{\rho^{t_2 - t_1}}{\eta\up{C} \eta\up{D}} C_{t_2} \epsilon $. Since $\mathbf{x}'$ cannot be better than $\mathbf{x}^*$, we derive
        \begin{equation} \label{eq:star_pd_pc}
            \eta\up{C} \eta\up{D} C_{t_1} \leq \rho^{t_2 - t_1} C_{t_2}.
        \end{equation}
    \end{itemize}

    Next, we build $\overline{\mathbf{x}}'$ based on $\overline{\mathbf{x}}$. We modify the solution at $t_1$, decreasing discharge if $\overline{p}\up{D}_{t_1}>0$ or increasing charge if $\overline{p}\up{D}_{t_1}=0$, which implies $\overline{p}\up{C}_{t_1}\geq0$. Therefore, the complementarity constraint at $t_1$ still holds.

    \textit{\textbf{If $\overline{p}\up{D}_{t_1}>0$:}} \\ 
    We decrease the discharge at $t_1$ by a small quantity $\overline{\epsilon} > 0$, with ${\epsilon} \leq \overline{p}\up{D}_{t_1}$, such that $\overline{p}\up{D'}_{t_1} = \overline{p}\up{D}_{t_1}-\overline{\epsilon}$ and the bounds on discharge are respected. 
    We apply \eqref{eq:s_calc} between $t_1$ and $t \in [t_1,t_2-1]$ to obtain the modified state of energy at $t$:
    \begin{equation}
        \overline{s}_t' = \rho^{t - t_1} \overline{s}_{t_1}' + \Delta t \sum_{k=t_1+1}^{t} \rho^{t - k} \left( \eta\up{C} \overline{p}_k\up{C'} - \frac{1}{\eta\up{D}} \overline{p}_k\up{D'} \right)  = \overline{s}_t + \rho^{t - t_1} \Delta t \frac{1}{\eta\up{D}} {\epsilon}.  
    \end{equation}
    We choose $\overline{\epsilon}$ such that the upper bound on the state of energy is respected, i.e. $\overline{\epsilon} \leq  \frac{\eta\up{D} (\overline{S} - \overline{s}_t)}{\rho^{t - t_1} \Delta t}$, $\forall t \in [t_1,t_2-1]$, which we know is possible since $\overline{s}_t < s_t^* \leq \overline{S}$, $ \forall t \in [t_1,t_2-1]$. 
    Since $\overline{s}_t \geq \underline{S}$, the lower bound is still respected.

    We then modify the solution at $t_2$ to return on the same trajectory, decreasing charge if $\overline{p}\up{C}_{t_2}>0$ or increasing discharge if $\overline{p}\up{C}_{t_2}=0$ and $\overline{p}\up{D}_{t_2}\geq0$. Therefore, the complementarity constraint at $t_2$ still holds.

    \begin{itemize}
        \item \textit{If $\overline{p}\up{C}_{t_2}>0$:} We decrease charge at $t_2$ by $\frac{\rho^{t_2 - t_1} \overline{\epsilon}}{\eta\up{C} \eta\up{D}}$, with $\overline{\epsilon} \leq \frac{\eta\up{C} \eta\up{D} \overline{p}\up{C}_{t_2}}{\rho^{t_2 - t_1}}$, such that $\overline{p}\up{C'}_{t_2}=\overline{p}\up{C}_{t_2} - \frac{\rho^{t_2 - t_1} \overline{\epsilon}}{\eta\up{C} \eta\up{D}}$ and the bounds on charge are respected.
        At $t_2$, the modified state of energy is
        \begin{equation}
            \overline{s}_{t_2}' = \rho \overline{s}_{t_2-1} + \rho^{t_2 - t_1} \Delta t \frac{1}{\eta\up{D}} \overline{\epsilon} + \Delta t \, \eta\up{C} \overline{p}\up{C}_{t_2} - \Delta t \, \eta\up{C} \frac{\rho^{t_2 - t_1} \overline{\epsilon}}{\eta\up{C} \eta\up{D}} - \Delta t \frac{1}{\eta\up{D}} \overline{p}\up{D}_{t_2} = \overline{s}_{t_2}.
        \end{equation}
        We are back on the same trajectory, which completes the demonstration that the new solution is feasible. The change in the objective function compared to $\mathbf{x}^*$ is $- \Delta t \, C_{t_1} \overline{\epsilon} + \Delta t \frac{\rho^{t_2 - t_1}}{\eta\up{C} \eta\up{D}} C_{t_2}  \overline{\epsilon}$. Since $\mathbf{x}'$ cannot be better than $\mathbf{x}^*$, we derive
        \begin{equation} \label{eq:overline_pd_pc}
            \eta\up{C} \eta\up{D} C_{t_1} \geq \rho^{t_2 - t_1} C_{t_2}.
        \end{equation}
        \item \textit{If $\overline{p}\up{C}_{t_2}=0$ and $\overline{p}\up{D}_{t_2}\geq 0$:} We increase discharge at $t_2$ by $\rho^{t_2 - t_1} \overline{\epsilon}$, with $\overline{\epsilon} \leq \frac{\overline{P}\up{D} - \overline{p}\up{D}_{t_2}}{\rho^{t_2 - t_1}}$, such that $\overline{p}\up{D'}_{t_2}=\overline{p}\up{D}_{t_2}+\rho^{t_2 - t_1}\overline{\epsilon}$. The bounds on discharge are respected. Note that $\overline{\epsilon} >0$ is guaranteed by $\overline{p}\up{D}_{t_2} < \overline{P}\up{D}$. Indeed, since $\overline{s}_{t_2} < {s}_{t_2}^*$, $\overline{s}_{t_2} \geq {s}_{t_2}^*$ would not be possible otherwise.
        At $t_2$, the state of energy is
        \begin{equation}
            \overline{s}_{t_2}' = \rho \overline{s}_{t_2-1} + \rho^{t_2 - t_1} \Delta t \frac{1}{\eta\up{D}} \overline{\epsilon} + \Delta t \, \eta\up{C} \overline{p}\up{C}_{t_2} - \Delta t \frac{1}{\eta\up{D}} \overline{p}\up{D}_{t_2} -  \Delta t \frac{1}{\eta\up{D}} \rho^{t_2 - t_1}\overline{\epsilon} =\overline{s}_{t_2}.  
        \end{equation}
        We are back on the same trajectory, which completes the demonstration that the new solution is feasible. The change in the objective function compared to $\mathbf{x}^*$ is $- \Delta t \, C_{t_1} \overline{\epsilon} + \Delta t \, \rho^{t_2 - t_1} C_{t_2} \overline{\epsilon}$. Since $\mathbf{x}'$ cannot be better than $\mathbf{x}^*$, we derive
        \begin{equation} \label{eq:overline_pd_pd}
            C_{t_1} \geq \rho^{t_2 - t_1} C_{t_2}.
        \end{equation}
    \end{itemize}

    \textit{\textbf{If $\overline{p}\up{D}_{t_1}=0$ and $\overline{p}\up{C}_{t_1}\geq0$:}} \\ 
    We increase the charge at $t_1$ by a small quantity $\overline{\epsilon} > 0$, with $\overline{\epsilon} \leq \overline{P}\up{C} - \overline{p}\up{C}_{t_1}$, such that $\overline{p}\up{C'}_{t_1}=\overline{p}\up{C}_{t_1}+\overline{\epsilon}$ and the bounds on charge are respected. Note that $\overline{\epsilon} >0$ is guaranteed by $\overline{p}\up{C}_{t_1} < \overline{P}\up{C}$. Indeed, since ${s}_{t_1}^* < \overline{s}_{t_1}$, ${s}_{t_1}^* \geq \overline{s}_{t_1}^*$ would not be possible otherwise.

    We apply \eqref{eq:s_calc} between $t_1$ and $t \in [t_1,t_2-1]$ to obtain the modified state of energy at $t$:
    \begin{equation}
        \overline{s}_t' = \rho^{t - t_1} \overline{s}_{t_1}' + \Delta t \sum_{k=t_1+1}^{t} \rho^{t - k} \left( \eta\up{C} \overline{p}_k\up{C'} - \frac{1}{\eta\up{D}} \overline{p}_k\up{D'} \right)  = \overline{s}_t + \rho^{t - t_1} \Delta t \, {\eta\up{C}} \overline{\epsilon}.
    \end{equation}
    We choose $\overline{\epsilon}$ such that the upper bound on the state of energy is respected, i.e. $\overline{\epsilon} \leq \frac{\overline{S} -\overline{s}_t}{\rho^{t - t_1} \Delta t \, {\eta\up{C}}}$, $\forall t \in [t_1,t_2-1]$, which is possible since $\overline{s}_t < s_t^* \leq \overline{S}$, $ \forall t \in [t_1,t_2-1]$. 
    Since $\overline{s}_t \geq \underline{S}$, the lower bound is still respected.

    We then modify the solution at $t_2$ in a similar way as before.
    \begin{itemize}
        \item \textit{If $\overline{p}\up{C}_{t_2}>0$: } We decrease charge at $t_2$ by $\rho^{t_2 - t_1} \overline{\epsilon}$, with $\overline{\epsilon} \leq \frac{\overline{p}\up{C}_{t_2}}{\rho^{t_2 - t_1}}$, such that $\overline{p}\up{C'}_{t_2}=\overline{p}\up{C}_{t_2}-\rho^{t_2 - t_1}\overline{\epsilon}$ and the bounds on charge are respected.
        At $t_2$, the state of energy is
        \begin{equation}
            \overline{s}_{t_2}' = \rho \overline{s}_{t_2-1} + \rho^{t_2 - t_1} \Delta t \, {\eta\up{C}} \overline{\epsilon} + \Delta t \, \eta\up{C} \overline{p}\up{C}_{t_2} - \Delta t \, \eta\up{C} \rho^{t_2 - t_1}\overline{\epsilon} - \Delta t \frac{1}{\eta\up{D}} \overline{p}\up{D}_{t_2} = \overline{s}_{t_2}.  
        \end{equation}
       We are back on the same trajectory, which completes the demonstration that the new solution is feasible. The change in the objective function compared to $\mathbf{x}^*$ is $- \Delta t \, C_{t_1} \overline{\epsilon} + \Delta t \, \rho^{t_2 - t_1} C_{t_2} \overline{\epsilon}$. Since $\mathbf{x}'$ cannot be better than $\mathbf{x}^*$, we derive
       \begin{equation} \label{eq:overline_pc_pc}
            C_{t_1} \geq \rho^{t_2 - t_1} C_{t_2}.
        \end{equation}
        \item \textit{If $\overline{p}\up{C}_{t_2}=0$ and $\overline{p}\up{D}_{t_2}\geq0$:} We increase discharge at $t_2$ by $\rho^{t_2 - t_1} {\eta\up{C} \eta\up{D}} \overline{\epsilon}$, with $\overline{\epsilon} \leq \frac{\overline{P}\up{D} - \overline{p}\up{D}_{t_2}}{\rho^{t_2 - t_1} \eta\up{C} \eta\up{D}}$, such that $\overline{p}\up{D'}_{t_2}=\overline{p}\up{D}_{t_2}+\rho^{t_2 - t_1} {\eta\up{C} \eta\up{D}}\overline{\epsilon}$. The bounds on discharge are respected.
        At $t_2$, the modified state of energy is 
        \begin{equation}
            \overline{s}_{t_2}' = \rho \overline{s}_{t_2-1} + \rho^{t_2 - t_1} \Delta t \, {\eta\up{C}} \overline{\epsilon} + \Delta t \, \eta\up{C} \overline{p}\up{C}_{t_2} - \Delta t \frac{1}{\eta\up{D}} \overline{p}\up{D}_{t_2} - \Delta t \frac{1}{\eta\up{D}} \rho^{t_2 - t_1} {\eta\up{C} \eta\up{D}} \overline{\epsilon} = \overline{s}_{t_2}.
        \end{equation}
        We are back on the same trajectory, which completes the demonstration that the new solution is feasible. The change in the objective function compared to $\mathbf{x}^*$ is $- \Delta t \, C_{t_1} \overline{\epsilon} + \Delta t \,\rho^{t_2 - t_1} C_{t_2} {\eta\up{C} \eta\up{D}} \overline{\epsilon}$. Since $\mathbf{x}'$ cannot be better than $\mathbf{x}^*$, we derive
        \begin{equation} \label{eq:overline_pc_pd}
            C_{t_1} \geq \rho^{t_2 - t_1} \eta\up{C} \eta\up{D} C_{t_2}.
        \end{equation}
    \end{itemize}
    
    To ensure that $s_{t_1-1}^* \leq \overline{s}_{t_1-1}$, $s_{t_1}^* > \overline{s}_{t_1}$, $s_{t_2-1}^* > \overline{s}_{t_2-1}$, and $s_{t_2}^* \leq \overline{s}_{t_2}$, the possible combinations of charge and discharge at $t_1$ are limited to the following:
    
    \begin{itemize}
        \item At $t_1$, $p_{t_1}\up{C*}>0$ and $\overline{p}_{t_1}\up{D}=0$. At $t_2$, $p_{t_2}\up{D*}=0$ and $\overline{p}_{t_2}\up{C}>0$.
        Then, with \eqref{eq:star_pc_pc} and \eqref{eq:overline_pc_pc}, we get
        \begin{equation} \label{eq:pc_pc_pc_pc}
            C_{t_1} = \rho^{t_2 - t_1} C_{t_2}.
        \end{equation}
        \item At $t_1$, $p_{t_1}\up{C*}>0$ and $\overline{p}_{t_1}\up{D}=0$. At $t_2$, $p_{t_2}\up{D*}>0$ and $\overline{p}_{t_2}\up{C}>0$.
        Then, with \eqref{eq:star_pc_pd} and \eqref{eq:overline_pc_pc}, we get
        \begin{equation} \label{eq:pc_pd_pc_pc}
            C_{t_1} = \rho^{t_2 - t_1} C_{t_2},
        \end{equation}
        and $\eta\up{C} \eta\up{D} = 1$ and/or $C_{t_1} = C_{t_2} = 0$.
        \item At $t_1$, $p_{t_1}\up{C*}>0$ and $\overline{p}_{t_1}\up{D}=0$. At $t_2$, $p_{t_2}\up{D*}>0$ and $\overline{p}_{t_2}\up{C}=0$.
        Then, with \eqref{eq:star_pc_pd} and \eqref{eq:overline_pc_pd}, we get
        \begin{equation} \label{eq:pc_pd_pc_pd}
            C_{t_1} = \rho^{t_2 - t_1} \eta\up{C} \eta\up{D} C_{t_2}.
        \end{equation}
        \item At $t_1$, $p_{t_1}\up{C*}>0$ and $\overline{p}_{t_1}\up{D}>0$. At $t_2$, $p_{t_2}\up{D*}=0$ and $\overline{p}_{t_2}\up{C}>0$.
        Then, with \eqref{eq:star_pc_pc} and \eqref{eq:overline_pd_pc}, we get
        \begin{equation} \label{eq:pc_pc_pd_pc}
            C_{t_1} = \rho^{t_2 - t_1} C_{t_2},
        \end{equation}
        and $\eta\up{C} \eta\up{D} = 1$ and/or $C_{t_1} = C_{t_2} = 0$.
        \item At $t_1$, $p_{t_1}\up{C*}>0$ and $\overline{p}_{t_1}\up{D}>0$. At $t_2$, $p_{t_2}\up{D*}>0$ and $\overline{p}_{t_2}\up{C}>0$.
        Then, with \eqref{eq:star_pc_pd} and \eqref{eq:overline_pd_pc}, we get
        \begin{equation} \label{eq:pc_pd_pd_pc}
            C_{t_1} =  \rho^{t_2 - t_1} C_{t_2},
        \end{equation}
        and $\eta\up{C} \eta\up{D} = 1$ and/or $C_{t_1} = C_{t_2} = 0$.
        \item At $t_1$, $p_{t_1}\up{C*}>0$ and $\overline{p}_{t_1}\up{D}>0$. At $t_2$, $p_{t_2}\up{D*}>0$ and $\overline{p}_{t_2}\up{C}=0$.
        Then, with \eqref{eq:star_pc_pd} and \eqref{eq:overline_pd_pd}, we get
        \begin{equation} \label{eq:pc_pd_pd_pd}
            C_{t_1} = \rho^{t_2 - t_1} C_{t_2},
        \end{equation}
        and $\eta\up{C} \eta\up{D} = 1$ and/or $C_{t_1} = C_{t_2} = 0$.
        \item At $t_1$, $p_{t_1}\up{C*}=0$ and $\overline{p}_{t_1}\up{D}>0$. At $t_2$, $p_{t_2}\up{D*}=0$ and $\overline{p}_{t_2}\up{C}>0$.
        Then, with \eqref{eq:star_pd_pc} and \eqref{eq:overline_pd_pc}, we get
        \begin{equation} \label{eq:pd_pc_pd_pc}
            \eta\up{C} \eta\up{D} C_{t_1} = \rho^{t_2 - t_1} C_{t_2}.
        \end{equation}
        \item At $t_1$, $p_{t_1}\up{C*}=0$ and $\overline{p}_{t_1}\up{D}>0$. At $t_2$, $p_{t_2}\up{D*}>0$ and $\overline{p}_{t_2}\up{C}>0$.
        Then, with \eqref{eq:star_pd_pd} and \eqref{eq:overline_pd_pc}, we get
        \begin{equation} \label{eq:pd_pd_pd_pc}
            C_{t_1} = C_{t_2} = \rho^{t_2 - t_1} C_{t_2},
        \end{equation}
        and $\eta\up{C} \eta\up{D} = 1$ and/or $C_{t_1} = C_{t_2} = 0$.
        \item At $t_1$, $p_{t_1}\up{C*}=0$ and $\overline{p}_{t_1}\up{D}>0$. At $t_2$, $p_{t_2}\up{D*}>0$ and $\overline{p}_{t_2}\up{C}=0$.
        Then, with \eqref{eq:star_pd_pd} and \eqref{eq:overline_pd_pd}, we get
        \begin{equation} \label{eq:pd_pd_pd_pd}
            C_{t_1} = \rho^{t_2 - t_1} C_{t_2}.
        \end{equation}
    \end{itemize}

    Focusing on $\mathbf{x}^*$, we have the following results:
    \begin{enumerate}
        \item If $p_{t_1}\up{C*}>0$ and $p_{t_2}\up{D*}=0$,
        \begin{equation} \label{eq:pc_pc}
            C_{t_1} = \rho^{t_2 - t_1} C_{t_2}.
        \end{equation}
        \item If $p_{t_1}\up{C*}>0$ and $p_{t_2}\up{D*}>0$, 
        \begin{equation} \label{eq:pc_pd}
            C_{t_1} = \rho^{t_2 - t_1} \eta\up{C} \eta\up{D} C_{t_2},
        \end{equation}
        since in \eqref{eq:pc_pd_pc_pc}, \eqref{eq:pc_pd_pd_pc} and \eqref{eq:pc_pd_pd_pd}, either $\eta\up{C} \eta\up{D} = 1$ and/or $C_{t_1} = C_{t_2} = 0$ so this result stands in any case.
        \item If $p_{t_1}\up{C*}=0$ and $p_{t_2}\up{D*}=0$,
        \begin{equation} \label{eq:pd_pc}
           \eta\up{C} \eta\up{D} C_{t_1} =  \rho^{t_2 - t_1} C_{t_2}.
        \end{equation}
        \item If $p_{t_1}\up{C*}=0$ and $p_{t_2}\up{D*}>0$,
        \begin{equation} \label{eq:pd_pd}
           C_{t_1} = \rho^{t_2 - t_1} C_{t_2}.
        \end{equation}
    \end{enumerate}

    Next, we iteratively build a new solution $\mathbf{x}'$ based on $\mathbf{x}^*$, and satisfying point 1 of the lemma, as follows:
    \begin{description}
        \item[$1\up{st}$ step:] Depending on the case, we modify the solution $\mathbf{x}^*$ at $t_1$ and $t_2$ to obtain $\mathbf{x}'$. We explain below how.
        \item[$2\up{nd}$ step:] For $\mathbf{x}'$, we identify $t_1$ and $t_2$, corresponding to the first interval for which the state of energy is strictly greater. If there is no $t_1$, then $s_t' \leq \overline{s}_t$, $\forall t \in \mathcal{T}$ and the procedure terminates.
        \item[$3\up{rd}$ step:] We update $\mathbf{x}^*$ to $\mathbf{x}'$ and go back to the $1\up{st}$ step.
    \end{description}

    We show that this procedure always converges and that the solution built is indeed optimal, which completes our proof.
    We consider separately the 4 cases identified previously.

    \textit{\textbf{1. $p_{t_1}\up{C*}>0$ and $p_{t_2}\up{D*}=0$.}} \\
    We modify the solution at $t_1$ to ${p}\up{C'}_{t_1}={p}\up{C*}_{t_1}-{\epsilon}$ and at $t_2$ to ${p}\up{C'}_{t_2}={p}\up{C*}_{t_2}+\rho^{t_2 - t_1} \epsilon$, with
    \begin{equation} \label{eq:epsilon_1}
        \epsilon = \min \left\{ {p}\up{C*}_{t_1}, \, \left\{\frac{{s}_t^* - \underline{S}}{\rho^{t - t_1} \Delta t \, \eta\up{C}}, \, \forall t \in [t_1,t_2-1]\right\}, \, \frac{\overline{P}\up{C} - p\up{C*}_{t_2}}{\rho^{t_2 - t_1}} \right\},
    \end{equation}
    which we saw previously is a feasible solution, with a change in the objective of
    \begin{equation}
        \Delta t \, C_{t_1} \epsilon - \Delta t \, \rho^{t_2 - t_1} C_{t_2} \epsilon = 0,
    \end{equation}
    with \eqref{eq:pc_pc}. Therefore, this new solution is also optimal.
    
    We show that for this solution, $s_t' > \overline{s}_t$ for fewer time periods than for $\mathbf{x}^*$.
    We evaluate separately the 3 possibilities:
    \begin{itemize}
        \item ${\epsilon} = {p}\up{C*}_{t_1}$. Then $p\up{C'}_t = 0$ and we move to case 3, for which we show below that we can further modify the solution in a way that $s_t' > \overline{s}_t$ for fewer time periods than for $\mathbf{x}^*$.
        \item ${\epsilon} = \min \{ \frac{s_t^* - \underline{S}}{\rho^{t - t_1} \Delta t \, \eta\up{C}}, \, t \in [t_1, t_2-1] \}$. Then for at least one $t \in [t_1, t_2-1]$, $s_t'= \underline{S} \leq \overline{s}_t$, and the result stands.
        \item ${\epsilon} = \frac{\overline{P}\up{C} - p\up{C*}_{t_2}}{\rho^{t_2 - t_1}}$. Then ${p}\up{C'}_{t_2} = \overline{P}\up{C}$. We have $s_{t_2}'- \rho s_{t_2-1}' = s_{t_2}^* - \rho  s_{t_2-1}' = \Delta t \, \eta\up{C} \overline{P}\up{C} \geq \overline{s}_{t_2} - \rho \overline{s}_{t_2-1}$. Since $s_{t_2}^* \leq \overline{s}_{t_2}$, it means that $s_{t_2-1}' \leq \overline{s}_{t_2-1}$ and the result stands.
    \end{itemize}

    \textit{\textbf{2. $p_{t_1}\up{C*}>0$ and $p_{t_2}\up{D*}>0$.}} \\
    We modify the solution at $t_1$ to ${p}\up{C'}_{t_1}={p}\up{C*}_{t_1}-{\epsilon}$ and at $t_2$ to ${p}\up{D'}_{t_2}={p}\up{D*}_{t_2}- \rho^{t_2 - t_1} \eta\up{C} \eta\up{D} \epsilon$, with
    \begin{equation} \label{eq:epsilon_2}
        \epsilon = \min \left\{ {p}\up{C*}_{t_1}, \, \left\{\frac{{s}_t^* - \underline{S}}{\rho^{t - t_1} \Delta t \, \eta\up{C}}, \, \forall t \in [t_1,t_2-1] \right\}, \, \frac{{p}\up{D*}_{t_2}}{\rho^{t_2 - t_1} \eta\up{C} \eta\up{D}} \right\},
    \end{equation}
    which we saw previously is a feasible solution, with a change in the objective of
    \begin{equation}
       \Delta t \, C_{t_1} \epsilon - \Delta t \, \rho^{t_2 - t_1} C_{t_2} \eta\up{C} \eta\up{D} \epsilon = 0,
    \end{equation}
    with \eqref{eq:pc_pd}. Therefore, this new solution is also optimal.
    
    We show that for this solution, $s_t' > \overline{s}_t$ for fewer time periods than for $\mathbf{x}^*$.
    We evaluate separately the 3 possibilities:
    \begin{itemize}
        \item ${\epsilon} = {p}\up{C*}_{t_1}$. Then $p\up{C'}_t = 0$ and we move to case 3, for which we show below that we can further modify the solution in a way that $s_t' > \overline{s}_t$ for fewer time periods than for $\mathbf{x}^*$.
        \item ${\epsilon} = \min \{ \frac{s_t^* - \underline{S}}{\rho^{t - t_1} \Delta t \, \eta\up{C}}, \, t \in [t_1, t_2-1] \}$. Then for at least one $t \in [t_1, t_2-1]$, $s_t'= \underline{S} \leq \overline{s}_t$, and the result stands.
        \item ${\epsilon} = \frac{{p}\up{D*}_{t_2}}{\rho^{t_2 - t_1} \eta\up{C} \eta\up{D}}$. Then ${p}\up{D'}_{t_2} = 0$ and we move to case 1, for which we have shown that we can further modify the solution in a way that $s_t' > \overline{s}_t$ for fewer time periods than for $\mathbf{x}^*$. 
    \end{itemize}

    \textit{\textbf{3. $p_{t_1}\up{C*}=0$ and $p_{t_2}\up{D*}=0$.}} \\
    We modify the solution at $t_1$ to ${p}\up{D'}_{t_1}={p}\up{D*}_{t_1}+{\epsilon}$ and at $t_2$ to ${p}\up{C'}_{t_2}={p}\up{C*}_{t_2}+\frac{\rho^{t_2 - t_1}}{\eta\up{C} \eta\up{D}}\epsilon$, with
        \begin{equation} \label{eq:epsilon_3}
            \epsilon = \min \left\{\overline{P}\up{D} - {p}\up{D*}_{t_1}, \, \left\{  \frac{\eta\up{D}({s}_t^* - \underline{S})}{\rho^{t - t_1} \Delta t}, \, \forall t \in [t_1,t_2-1] \right\} , \, \frac{\eta\up{C} \eta\up{D} (\overline{P}\up{C} - p\up{C*}_{t_2})}{\rho^{t_2 - t_1}} \right\},
        \end{equation}
        which we saw previously is a feasible solution, with a change in the objective of
        \begin{equation}
            \Delta t \, C_{t_1} \epsilon - \Delta t \, C_{t_2} \frac{\rho^{t_2 - t_1}}{\eta\up{C} \eta\up{D}} \epsilon = 0,
        \end{equation}
        with \eqref{eq:pd_pc}. Therefore, this new solution is also optimal.
        
        We show that for this solution, $s_t' > \overline{s}_t$ for fewer time periods than for $\mathbf{x}^*$.
        We evaluate separately the 3 possibilities:
        \begin{itemize}
            \item ${\epsilon} = \overline{P}\up{D} - {p}\up{D*}_{t_1}$. Then $p\up{D'}_{t_1} = \overline{P}\up{D}$. We have $s_{t_1}'- \rho s_{t_1-1}' = s_{t_1}' - \rho s_{t_1-1}^* = - \Delta t \frac{1}{\eta\up{D}} \overline{P}\up{D} \leq \overline{s}_{t_1} - \rho \overline{s}_{t_1-1}$. Since $s_{t_1-1}^* \leq \overline{s}_{t_1-1}$, it means that $s_{t_1}' \leq \overline{s}_{t_1}$ and the result stands.
            \item ${\epsilon} = \min \{ \eta\up{D} \frac{{s}_t^* - \underline{S}}{\rho^{t - t_1} \Delta t}, \, t \in [t_1, t_2-1] \}$. Then for at least one $t \in [t_1, t_2-1]$, $s_t'= \underline{S} \leq \overline{s}_t$, and the result stands.
            \item ${\epsilon} = \frac{\eta\up{C} \eta\up{D} (\overline{P}\up{C} - p\up{C*}_{t_2})}{\rho^{t_2 - t_1}}$. Then ${p}\up{C'}_{t_2} = \overline{P}\up{C}$. We have $s_{t_2}'- \rho s_{t_2-1}' = s_{t_2}^* - \rho s_{t_2-1}' = \Delta t \, \eta\up{C} \overline{P}\up{C} \geq \overline{s}_{t_2} - \rho \overline{s}_{t_2-1}$. Since $s_{t_2}^* \leq \overline{s}_{t_2}$, it means that $s_{t_2-1}' \leq \overline{s}_{t_2-1}$ and the result stands.
        \end{itemize}

    \textit{\textbf{4. $p_{t_1}\up{C*}=0$ and $p_{t_2}\up{D*}>0$.}} \\
    We modify the solution at $t_1$ to ${p}\up{D'}_{t_1}={p}\up{D*}_{t_1}+{\epsilon}$ and at $t_2$ to ${p}\up{D'}_{t_2}={p}\up{D*}_{t_2}-\rho^{t_2 - t_1} \epsilon$, with
        \begin{equation} \label{eq:epsilon_4}
            \epsilon = \min \left\{\overline{P}\up{D} - {p}\up{D*}_{t_1}, \, \left\{ \eta\up{D} \frac{{s}_t^* - \underline{S}}{\rho^{t - t_1} \Delta t}, \, \forall t \in [t_1,t_2-1] \right\} , \, \frac{p\up{D*}_{t_2}}{\rho^{t_2 - t_1}} \right\},
        \end{equation}
        which we saw previously is a feasible solution, with a change in the objective of
        \begin{equation}
            \Delta t \, C_{t_1} \epsilon - \Delta t \, \rho^{t_2 - t_1} C_{t_2} \epsilon = 0,
        \end{equation}
        with \eqref{eq:pd_pd}. Therefore, this new solution is also optimal.
        
        We show that for this solution, $s_t' > \overline{s}_t$ for fewer time periods than for $\mathbf{x}^*$.
        We evaluate separately the 3 possibilities:
        \begin{itemize}
            \item ${\epsilon} = \overline{P}\up{D} - {p}\up{D*}_{t_1}$. Then $p\up{D'}_{t_1} = \overline{P}\up{D}$. We have $s_{t_1}'- \rho s_{t_1-1}' = s_{t_1}' - \rho s_{t_1-1}^* = - \Delta t \frac{1}{\eta\up{D}} \overline{P}\up{D} \leq \overline{s}_{t_1} - \rho \overline{s}_{t_1-1}$. Since $s_{t_1-1}^* \leq \overline{s}_{t_1-1}$, it means that $s_{t_1}' \leq \overline{s}_{t_1}$ and the result stands.
            \item ${\epsilon} = \min \{ \eta\up{D} \frac{{s}_t^* - \underline{S}}{\Delta t}, \, t \in [t_1, t_2-1] \}$. Then for at least one $t \in [t_1, t_2-1]$, $s_t'= \underline{S} \leq \overline{s}_t$, and the result stands.
            \item ${\epsilon} = \frac{p\up{D*}_{t_2}}{\rho^{t_2 - t_1}}$. Then ${p}\up{D'}_{t_2} = 0$ and we move to case 3, for which we have shown that we can further modify the solution in a way that $s_t' > \overline{s}_t$ for fewer time periods than for $\mathbf{x}^*$.
        \end{itemize}

    By induction, starting again the procedure on the new solution, with updated $t_1$ and $t_2$, we can build a solution $\mathbf{x}' \in \mathcal{X}^*$ for which $s_t' \leq \overline{s}_t$, $\forall t \in \mathcal{T}$, which proves point 1 of the lemma.
    
    The proof for point 2 of the lemma is similar.
\end{proof}

Based on Lemma~\ref{lemma1}, the following can also be derived:
\begin{corollary} \label{corol1}
    For any $\overline{\mathbf{x}} \in \overline{\mathcal{X}}$ and $\underline{\mathbf{x}} \in \underline{\mathcal{X}}$, such that $\underline{s}_t \leq \overline{s}_t$, $\forall t \in \mathcal{T}$, and for any $S\up{end}$, $\underline{S}_T \leq S\up{end} \leq \overline{S}_T$, $\exists \, \mathbf{x}^* \in \mathcal{X}^*$ such that $\underline{s}_t \leq s_t^* \leq \overline{s}_t$, $\forall t \in \mathcal{T}$.
\end{corollary}

\begin{proof}
    First, given $\overline{\mathbf{x}} \in \overline{\mathcal{X}}$, we know from Lemma~\ref{lemma1} that there exists $\underline{\mathbf{x}} \in \underline{\mathcal{X}}$, such that $\underline{s}_t \leq \overline{x}_t$, $\forall t \in \mathcal{T}$.
    For the second part, using Lemma~\ref{lemma1}, we know that there exists $\mathbf{x}^* \in \mathcal{X}^*$ such that ${s}_t^* \geq \underline{s}_t$, $\forall t \in \mathcal{T}$.
    Supposing that there exists $t \in \mathcal{T}$ such that $s_{t}^* > \overline{s}_t$, we can build a new solution $\mathbf{x}' \in \mathcal{X}^*$ such that $\underline{s}_t \leq s_t' \leq \overline{x}_t$, $\forall t \in \mathcal{T}$, similarly to what was done in the proof of Lemma~\ref{lemma1}.
    The only difference is that the $\epsilon$ used to modify the solution should be such that it stays above $\underline{s}_t$ instead of $\underline{S}$.
    For example, \eqref{eq:epsilon_1} becomes
    \begin{equation}
        \epsilon = \min \left\{ {p}\up{C*}_{t_1}, \, \left\{\frac{{s}_t^* - \underline{s}_t}{\rho^{t - t_1} \Delta t \, \eta\up{C}}, \, \forall t \in [t_1,t_2-1]\right\}, \, \frac{\overline{P}\up{C} - p\up{C*}_{t_2}}{\rho^{t_2 - t_1}} \right\}
    \end{equation}
    where $\underline{S}$ has been replaced by $\underline{s}_t$. 
    The same goes for \eqref{eq:epsilon_2}, \eqref{eq:epsilon_3} and \eqref{eq:epsilon_4}.
\end{proof}

We also need the following result:

\begin{lemma} \label{lemma2}
    If there exist $\overline{\mathbf{x}} \in \overline{\mathcal{X}}$, $\underline{\mathbf{x}} \in \underline{\mathcal{X}}$ and $ \tau \in \mathcal{T}$ such that $\underline{s}_\tau = \overline{s}_\tau$, then there exist $\overline{\mathbf{x}} \in \overline{\mathcal{X}}$ and $\underline{\mathbf{x}} \in \underline{\mathcal{X}}$ such that $\underline{s}_t = \overline{s}_t$, $\forall t \in [1,\tau]$ and $\underline{s}_t \leq \overline{s}_t$, $\forall t \in [\tau+1, T]$.
\end{lemma}

\begin{proof}
    Let's identify as $S$ the state of energy at $\tau$ for which $\underline{s}_\tau = \overline{s}_\tau = S$.
    Augmenting problem $\textbf{F}(\mathcal{T}, \textbf{C}, \underline{S}_T)$ with the constraint $s_\tau = S$ will return the same optimal objective value. For this augmented problem, we can solve independently for $t \in [1,\tau]$ and for $t \in [\tau+1,T]$, since in the only constraint linking those two periods, $s_{\tau+1} = \rho s_\tau +  \Delta t \left(\eta\up{C} p_\tau\up{C} - \frac{1}{\eta\up{D}} p_\tau\up{D}\right)$, $s_\tau$ can be replaced by $S$.
    Similarly, augmenting $\textbf{F}(\mathcal{T}, \textbf{C}, \overline{S}_T)$ with the constraint $s_\tau = S$ will return the same optimal objective value, and for this augmented problem, we can solve independently for $t \in [1,\tau]$ and for $t \in [\tau+1,T]$.
    The two problems are the same over the first period, i.e. for $t \in [1,\tau]$, so their optimal solutions are the same.
    Over the second period, as the starting point is the same, the results from Lemma~\ref{lemma1} can be applied for $t \in [\tau+1,T]$ to show that there exists $\underline{\mathbf{x}} \in \underline{\mathcal{X}}$, such that $\underline{s}_t \leq \overline{x}_t$, $\forall t \in [\tau+1,T]$.
    Combining the solutions over the two periods, we get the desired result.
\end{proof}

With Lemma~\ref{lemma2}, if $\exists \, \underline{\mathbf{x}} \in \underline{\mathcal{X}}$ and $\exists \, \overline{\mathbf{x}} \in \overline{\mathcal{X}}$ such that $\underline{s}_H = \overline{s}_H$, then $\exists \, \overline{\mathbf{x}}' \in \overline{\mathcal{X}}$ and $\exists \, \underline{\mathbf{x}}' \in \underline{\mathcal{X}}$ such that $\underline{s}'_t = \overline{s}'_t$, $\forall t \in [1,\tau]$ and $\underline{s}'_t \leq \overline{s}'_t$, $\forall t \in [1,\tau]$.
With Corollary~\ref{corol1}, for any $S\up{end}$, $\underline{S}_T \leq S\up{end} \leq \overline{S}_T$, $\exists \, \mathbf{x}^* \in \mathcal{X}^*$ such that $\underline{s}'_t \leq s_t^* \leq \overline{s}'_t$, $\forall t \in \mathcal{T}$, and in particular, such that $\underline{s}'_t = s_t^* = \overline{s}'_t$, $\forall t \in [1,H]$.

Solving the problems over a horizon longer by one time period, $T+1$, the optimal level at $T$, called $S$, obtained from solving $\textbf{F}([1,T+1], \textbf{C}, \underline{S}_{T+1})$ is such that $\underline{S}_T \leq S \leq \overline{S}_T$, since $\underline{S}_T$ and $\overline{S}_T$ are the lowest and highest reachable levels at $T$. We can add the constraint $s_T = S$ and split the problem over the two intervals, $[1,T]$ and $[T+1]$ without changing the optimal value of the objective function. The problem on $[1,T]$ is then $\textbf{F}(\mathcal{T}, \textbf{C}, {S})$, for which we have already shown that the solution over the decision horizon will be the same.
The same applies for $\textbf{F}([1,T+1], \textbf{C}, \overline{S}_{T+1})$, and therefore for $\textbf{F}([1,T+1], \textbf{C}, {S}\up{end})$, with any ${S}\up{end}$, such that  $\underline{S}_{T+1} \leq {S}\up{end} \leq \overline{S}_{T+1}$.
By induction, we obtain that the solution over the decision horizon is optimal for the infinite horizon problem.

On the other hand, if $T$ is a forecast horizon, by definition, $\exists \, \underline{\mathbf{x}} \in \underline{\mathcal{X}}$ and $\exists \, \overline{\mathbf{x}} \in \overline{\mathcal{X}}$ such that $\underline{s}_H = \overline{s}_H$.
\section{Proof of Proposition~\ref{prop2}} \label{app:proof3}

For a given $s_{H} \in [\underline{s}_H,\overline{s}_H]$, the difference in objective value with respect to the solution of the infinite-horizon problem is the sum of the difference due to the divergent schedule over the decision horizon and the difference due to starting the remainder of the horizon with a different storage level:
\begin{equation}
    \Delta Z = Z\up{*,DH} - Z\up{DH} + Z\up{*,R} - Z\up{R}
\end{equation}
where $Z\up{*,DH}$ is the value of the objective function over the decision horizon for the infinite-horizon problem, and $Z\up{*,R}$ and $Z\up{R}$ are the values of the objective function over the rest of the horizon for the infinite-horizon problem and for the problem actually solved, respectively.

It is possible to bound $Z\up{*,DH}$ from above with $Z\up{opt,DH}$, the maximum profit obtainable if considering the decision horizon in isolation and allowing the final storage at the end of the decision horizon to be a variable bounded between $\underline{s}_H$ and $\overline{s}_H$, the interval within which the optimal storage level in the infinite horizon problem lies. 

Over the rest of the horizon, we can define two extreme cases.  
In the first case, $s_H \geq s_H^*$, where $s_H^*$ is the optimal level at the end of the decision horizon in the infinite-horizon problem, and we miss an opportunity to charge the difference later at a lower price, and in the worst case at $\underline{C}$. Therefore we have 
\begin{equation} \label{eq:subopt_high}
    Z\up{*,R} - Z\up{R} \leq - \underline{C} \frac{1}{\eta\up{C}} (s_H - s_H^*) \leq - \underline{C} \frac{1}{\eta\up{C}} (s_H - \underline{s}_H).
\end{equation}
In the second case, $s_H \leq s_H^*$, and we miss an opportunity to discharge the difference $s_H^*-s_H$ later at a higher price, and in the worst case at $\overline{C}$. Therefore we have 
\begin{equation} \label{eq:subopt_low}
    Z\up{*,R} - Z\up{R} \leq \overline{C} \eta\up{D} (s_H^* - s_H) \leq \overline{C} \eta\up{D} (\overline{s}_H - s_H).
\end{equation}

Combining \eqref{eq:subopt_high} and \eqref{eq:subopt_low} gives
\begin{equation} \label{eq:subopt1_proof}
    \Delta Z \leq Z\up{opt,DH} - Z\up{DH} + \max \{- \underline{C} \frac{1}{\eta\up{C}} (s_H - \underline{s}_H),  \overline{C} \eta\up{D} (\overline{s}_H - s_H)\}.
\end{equation}
\section{Proof of Proposition~\ref{prop1}} \label{app:proof2}

We prove the result by contradiction. Consider $T \in \mathbb{N}^+$, such that the expression in \eqref{eq:cond2} is strictly positive and suppose that $T$ is a forecast horizon. 
We then have the following:
\begin{equation} \label{eq:proof1}
    \overline{S}  - \underline{S} - \left ( \sum_{t=0}^{T-H-1} \rho^t \right ) \left(\Delta t \, \eta\up{C} \overline{P}\up{C} + \frac{\Delta t}{\eta\up{D}} \overline{P}\up{D} \right) > 0,
\end{equation}
\begin{equation} \label{eq:proof2}
    \rho^{T} S\up{init} - \underline{S} + \Delta t \, \eta\up{C}  \overline{P}\up{C} \left ( \sum_{t=T-H}^{T-1} \rho^t \right ) - \frac{\Delta t}{\eta\up{D}} \overline{P}\up{D} \left ( \sum_{t=0}^{T-H-1} \rho^t \right ) > 0,
\end{equation}
and
\begin{equation} \label{eq:proof3}
     \overline{S} - \rho^{T} S\up{init} -  \Delta t \, \eta\up{C}  \overline{P}\up{C} \left ( \sum_{t=0}^{T-H-1} \rho^t \right ) + \frac{\Delta t}{\eta\up{D}} \overline{P}\up{D} \left ( \sum_{t=T-H}^{T-1} \rho^t \right ) > 0,
\end{equation}
Since $T$ is a forecast horizon, with Theorem~\ref{theorem1}, $\exists \, \underline{\mathbf{x}} \in \underline{\mathcal{X}}$ and $\exists \, \overline{\mathbf{x}} \in \overline{\mathcal{X}}$ such that $\underline{s}_H = \overline{s}_H$. We consider such $\underline{\mathbf{x}}$ and $\overline{\mathbf{x}}$, and we introduce ${S}$ to represent this value, i.e. ${S}=\underline{s}_H =\overline{s}_H$. Tight lower and upper bounds on ${S}$ can be obtained by considering that the maximum quantity is discharged during the whole decision horizon or until the minimum level is reached:
\begin{equation} \label{eq:proof4}
    {S} \geq \max \left \{\underline{S}, \rho^{H} S\up{init} - \frac{\Delta t}{\eta\up{D}} \overline{P}\up{D} \sum_{t=0}^{H-1} \rho^t \right \},
\end{equation}
and by considering that the maximum quantity is discharged during the whole decision horizon or until the minimum level is reached:
\begin{equation} \label{eq:proof5}
    {S} \leq \min \left \{\overline{S}, \rho^{H} S\up{init} + \Delta t \, \eta\up{C}  \overline{P}\up{C} \sum_{t=0}^{H-1} \rho^t \right \}.
\end{equation}

Moreover, we know that the minimum and the maximum reachable levels at the end of the planning horizon, $\underline{S}_T$ and $\overline{S}_T$, can be reached from ${S}$. In other words, if discharging the maximum quantity after the end of the decision horizon and until the end of the planning horizon $\underline{S}_T$ should at least be reached:
\begin{equation} \label{eq:proof6}
    \rho^{T-H} {S} -  \frac{\Delta t}{\eta\up{D}} \overline{P}\up{D} \left ( \sum_{t=0}^{T-H-1} \rho^t \right ) \leq \underline{S}_T.
\end{equation}
Similarly, if charging the maximum quantity after the end of the decision horizon and until the end of the planning horizon, $\overline{S}_T$ should at least be reached:
\begin{equation} \label{eq:proof7}
    \rho^{T-H} {S} + \Delta t \eta\up{C} \overline{P}\up{C} \left ( \sum_{t=0}^{T-H-1} \rho^t \right ) \geq \overline{S}_T.
\end{equation}

Next we use the fact that $\underline{S}_T = \max \left \{\underline{S}, \rho^{T} S\up{init} - \frac{\Delta t}{\eta\up{D}} \overline{P}\up{D} \sum_{t=0}^{T-1} \rho^t \right \}$ and $\overline{S}_T = \min\left \{\overline{S}, \rho^{T} S\up{init} +\right.$ $ \left. \Delta t \, \eta\up{C} \overline{P}\up{C} \sum_{t=0}^{T-1} \rho^t \right \}$ to define the four following cases:
\begin{enumerate}
    \item $\underline{S}_T = \underline{S}$ and $\overline{S}_T = \overline{S}$ 
    \item $\underline{S}_T = \underline{S}$ and $\overline{S}_T = \rho^{T} S\up{init} + \Delta t \eta\up{C} \overline{P}\up{C} \sum_{t=0}^{T-1} \rho^t$ 
    \item $\underline{S}_T = \rho^{T} S\up{init} - \frac{\Delta t}{\eta\up{D}} \overline{P}\up{D} \sum_{t=0}^{T-1} \rho^t$ and $\overline{S}_T = \overline{S}$ 
    \item $\underline{S}_T = \rho^{T} S\up{init} - \frac{\Delta t}{\eta\up{D}} \overline{P}\up{D} \sum_{t=0}^{T-1} \rho^t$ and $\overline{S}_T = \rho^{T} S\up{init} + \Delta t \, \eta\up{C} \overline{P}\up{C} \sum_{t=0}^{T-1} \rho^t$
\end{enumerate}

In Cases 1 and 2, \eqref{eq:proof6} becomes
\begin{equation}
    \rho^{T-H} {S} -  \frac{\Delta t}{\eta\up{D}} \overline{P}\up{D} \left ( \sum_{t=0}^{T-H-1} \rho^t \right ) \leq \underline{S},
\end{equation}
which can be rearranged as
\begin{equation} \label{eq:proof8}
    {S} \leq \rho^{H-T} \underline{S} + \frac{\Delta t}{\eta\up{D}} \overline{P}\up{D} \rho^{H-T} \left ( \sum_{t=0}^{T-H-1} \rho^t \right ).
\end{equation}

In Cases 3 and 4, \eqref{eq:proof6} becomes
\begin{equation}
    \rho^{T-H} {S} -  \frac{\Delta t}{\eta\up{D}} \overline{P}\up{D} \left ( \sum_{t=0}^{T-H-1} \rho^t \right ) \leq  \rho^{T} S\up{init} - \frac{\Delta t}{\eta\up{D}} \overline{P}\up{D} \sum_{t=0}^{T-1} \rho^t,
\end{equation}
which can be rearranged as
\begin{equation}
    {S} \leq \rho^{H}S\up{init} - \frac{\Delta t}{\eta\up{D}} \overline{P}\up{D} \sum_{t=0}^{H-1} \rho^t.
\end{equation}
With \eqref{eq:proof4}, we get that 
\begin{equation} \label{eq:proof9}
    {S} = \rho^{H}S\up{init} - \frac{\Delta t}{\eta\up{D}} \overline{P}\up{D} \sum_{t=0}^{H-1} \rho^t.
\end{equation}

In Cases 1 and 3, \eqref{eq:proof7} becomes
\begin{equation}
    \rho^{T-H} {S} + \Delta t \, \eta\up{C} \overline{P}\up{C} \left ( \sum_{t=0}^{T-H-1} \rho^t \right ) \geq \overline{S},
\end{equation}
which can be rearranged as
\begin{equation} \label{eq:proof10}
    {S} \geq \rho^{H-T} \overline{S} - \Delta t \, \eta\up{C} \overline{P}\up{C} \rho^{H-T} \left ( \sum_{t=0}^{T-H-1} \rho^t \right ).
\end{equation}

In Cases 2 and 4, \eqref{eq:proof7} becomes
\begin{equation}
    \rho^{T-H} {S} + \Delta t \, \eta\up{C} \overline{P}\up{C} \left ( \sum_{t=0}^{T-H-1} \rho^t \right ) \geq \rho^{T} S\up{init} + \Delta t \, \eta\up{C} \overline{P}\up{C} \sum_{t=0}^{T-1} \rho^t,
\end{equation}
which can be rearranged as
\begin{equation}
    {S} \geq \rho^{H} S\up{init} + \Delta t \, \eta\up{C}  \overline{P}\up{C} \sum_{t=0}^{H-1} \rho^t.
\end{equation}
With \eqref{eq:proof5}, we get that 
\begin{equation} \label{eq:proof11}
    {S} = \rho^{H} S\up{init} + \Delta t \, \eta\up{C}  \overline{P}\up{C} \sum_{t=0}^{H-1} \rho^t.
\end{equation}

\begin{description}
    \item[Case 1:] In the first case, \eqref{eq:proof8} and \eqref{eq:proof10} apply. This is possible only if 
    \begin{equation}
    \rho^{H-T} \overline{S} - \Delta t \, \eta\up{C} \overline{P}\up{C} \rho^{H-T} \left ( \sum_{t=0}^{T-H-1} \rho^t \right ) \leq \rho^{H-T} \underline{S} + \frac{\Delta t}{\eta\up{D}} \overline{P}\up{D} \rho^{H-T} \left ( \sum_{t=0}^{T-H-1} \rho^t \right ).
    \end{equation}
    We can rearrange as 
    \begin{equation}
        \overline{S} - \underline{S}  - \left ( \sum_{t=0}^{T-H-1} \rho^t \right ) \left(\Delta t \, \eta\up{C} \overline{P}\up{C} + \frac{\Delta t}{\eta\up{D}} \overline{P}\up{D} \right) \leq 0,
    \end{equation} 
    which is not possible because of \eqref{eq:proof1}.
    \item[Case 2:] In the second case, \eqref{eq:proof8} and \eqref{eq:proof11} apply. Using \eqref{eq:proof11}, we can replace in \eqref{eq:proof8}:
    \begin{equation}
        \rho^{H} S\up{init} + \Delta t \, \eta\up{C}  \overline{P}\up{C} \sum_{t=0}^{H-1} \rho^t \leq \rho^{H-T} \underline{S} + \frac{\Delta t}{\eta\up{D}} \overline{P}\up{D} \rho^{H-T} \left ( \sum_{t=0}^{T-H-1} \rho^t \right ).
    \end{equation}
    Rearranging gives 
    \begin{equation}
        \rho^{T} S\up{init} - \underline{S} + \Delta t \, \eta\up{C}  \overline{P}\up{C} \left ( \sum_{t=T-H}^{T-1} \rho^t \right ) - \frac{\Delta t}{\eta\up{D}} \overline{P}\up{D} \left ( \sum_{t=0}^{T-H-1} \rho^t \right ) \leq 0,
    \end{equation} 
    which is not possible because of \eqref{eq:proof2}.
    \item[Case 3:] In the third case, \eqref{eq:proof9} and \eqref{eq:proof10} apply. Using \eqref{eq:proof9}, we can replace in \eqref{eq:proof10}:
    \begin{equation}
        \rho^{H}S\up{init} - \frac{\Delta t}{\eta\up{D}} \overline{P}\up{D} \sum_{t=0}^{H-1} \rho^t \geq \rho^{H-T} \overline{S} - \Delta t \, \eta\up{C} \overline{P}\up{C} \rho^{H-T} \left ( \sum_{t=0}^{T-H-1} \rho^t \right ).
    \end{equation}
    Rearranging gives 
    \begin{equation}
        \overline{S} - \rho^{T} S\up{init} -  \Delta t \, \eta\up{C}  \overline{P}\up{C} \left ( \sum_{t=0}^{T-H-1} \rho^t \right ) + \frac{\Delta t}{\eta\up{D}} \overline{P}\up{D} \left ( \sum_{t=T-H}^{T-1} \rho^t \right ) \leq 0,
    \end{equation} 
    which is not possible because of \eqref{eq:proof3}.
    \item[Case 4:] In the fourth case, \eqref{eq:proof9} and \eqref{eq:proof11} apply. We thus have:
    \begin{equation}
        \rho^{H}S\up{init} - \frac{\Delta t}{\eta\up{D}} \overline{P}\up{D} \sum_{t=0}^{H-1} \rho^t = \rho^{H} S\up{init} + \Delta t \, \eta\up{C}  \overline{P}\up{C} \sum_{t=0}^{H-1} \rho^t,
    \end{equation}
    which means that
    \begin{equation}
         - \frac{1}{\eta\up{D}} \overline{P}\up{D} =  \eta\up{C} \overline{P}\up{C},
    \end{equation}
    which is not possible since all are strictly positive.
\end{description}

We showed that in all the cases there is a contradiction, therefore, $T$ is not a forecast horizon.

\end{appendices}

\bibliographystyle{elsarticle-num}
\bibliography{ref}

\end{document}